\documentclass[a4paper,11pt,twoside]{amsart}

\usepackage{amsmath}
\usepackage{amsthm}
\usepackage{amssymb}
\usepackage[all]{xy}
\usepackage{hyperref}

\newtheorem{theorem}{Theorem}[section]
\newtheorem{proposition}[theorem]{Proposition}
\newtheorem{lemma}[theorem]{Lemma}
\newtheorem{corollary}[theorem]{Corollary}

\newtheorem*{blank}{Theorem \ref{main-theorem1}}

\numberwithin{equation}{section}

\theoremstyle{definition}
\newtheorem{definition}[theorem]{Definition}
\newtheorem{remark}[theorem]{Remark}
\newtheorem{example}[theorem]{Example}

\newcommand{\Db}{{\rm D}^{\rm b}}

\newcommand{\Br}{{\rm Br}}

\newcommand{\Pic}{{\rm Pic}}
\newcommand{\Prym}{{\rm Prym}}
\newcommand{\ch}{{\rm ch}}
\newcommand{\Td}{{\rm Td}}

\newcommand{\coh}{{\cat{Coh}}}

\newcommand{\Hom}{{\rm Hom}}
\newcommand{\Spec}{{\rm Spec}}

\newcommand{\corr}{{\rm Corr}}

\newcommand{\id}{{\rm Id}}

\newcommand{\cat}[1]{\begin{bf}#1\end{bf}}

\newcommand{\cal}{\mathcal}
\newcommand{\ka}{{\cal A}}
\newcommand{\kb}{{\cal B}}
\newcommand{\kc}{{\cal C}}

\newcommand{\ke}{{\cal E}}
\newcommand{\kf}{{\cal F}}
\newcommand{\kg}{{\cal G}}

\newcommand{\km}{{\cal M}}
\newcommand{\ko}{{\cal O}}

\newcommand{\kz}{{\cal Z}}

\newcommand{\ZZ}{\mathbb{Z}}
\newcommand{\QQ}{\mathbb{Q}}

\newcommand{\CC}{\mathbb{C}}

\newcommand{\FF}{\mathbb{F}}
\newcommand{\PP}{\mathbb{P}}

\def\ee{\varepsilon}

\def\cI{\mathcal{I}}

\begin{document}

\title{Derived categories and rationality of conic bundles}

\author[M.\ Bernardara, M.\ Bolognesi]{Marcello Bernardara and Michele Bolognesi}

\address{M.Be.: Univerist\"at Duisburg--Essen, Fakult\"at f\"ur Mathematik. Universit\"atstr. 2, 45117 Essen (Germany),
\& Institut de Math\'{e}matiques de Toulouse (IMT), 118 route de Narbonne, F-31062 Toulouse Cedex 9, (France)}
\email{mbernard@math.univ-toulouse.fr}

\address{M.Bo.: IRMAR, Universit\'e de Rennes 1.
263 Av. G\'en\'eral Leclerc, 35042 Rennes CEDEX (France)} 
\email{michele.bolognesi@univ-rennes1.fr}

\begin{abstract}
We show that a standard conic bundle over a minimal rational surface is rational
and its Jacobian splits as the direct sum of Jacobians of curves if and
only if its derived category admits a semiorthogonal decomposition by exceptional objects and
the derived categories of those curves. Moreover, such a decomposition gives the splitting 
of the intermediate Jacobian also when the surface is not minimal.
\end{abstract}

\maketitle
\section{Introduction}

One of the main fields of research in the theory of derived categories is understanding how the geometry
of a smooth projective variety $X$ is encoded in the bounded derived category $\Db(X)$ of coherent sheaves on it.
One of the main ideas, first developed by Bondal and Orlov, is to understand to which extent this category
contains interesting information about birational geometry.

The biggest problem is to understand how this information can be traced out. The most promising and,
so far, prolific approach is studying semiorthogonal decompositions

$$\Db(X) = \langle \cat{A}_1, \ldots, \cat{A}_k \rangle.$$

In many interesting situations, one has such a decomposition with all or almost all of the $\cat{A}_i$ equivalent to
the derived category of a point. If $X$ is a projective space or a smooth quadric, all of the $\cat{A}_i$
are like this. It is expected that if a non-trivial subcategory appears in such decomposition, then
it has to carry informations about the birational geometry of $X$. For example, if $X$ is a $V_{14}$ Fano
threefold, then $\Db(X)$ admits a semiorthogonal decomposition with only one non-trivial component, say
$\cat{A}_X$. A similar decomposition holds for any smooth cubic threefold.
Kuznetsov showed that if $Y$ is the unique cubic
threefold birational to $X$ (see \cite{iliev-marku}), $\cat{A}_X$ is equivalent to the non-trivial component $\cat{A}_Y$
of $\Db(Y)$, and then it is a birational
invariant for $X$ \cite{kuznet-v14}. Moreover it has been shown in \cite{noi-cubic}, by reconstructing
the Fano variety of lines on $Y$ from $\cat{A}_Y$, that $\cat{A}_Y$ determines the isomorphism class of $Y$.
Similar correspondences between the non-trivial components of semiorthogonal decompositions of pairs of
Fano threefolds are described in \cite{kuznet-fano3folds}. The derived category of a smooth cubic fourfold
also admits such a decomposition, and it is conjectured that the non-trivial component determines its
rationality \cite{kuznet-cubic4folds}. 

It is a classical and still open problem in complex algebraic geometry to study the rationality of a standard
conic bundle $\pi: X \to S$
over a smooth projective surface. A necessary condition for rationality is that the intermediate Jacobian
$J(X)$ is isomorphic, as principally polarized Abelian variety, to the direct sum of Jacobians of smooth
projective curves. This allowed to prove the non rationality of smooth cubic threefolds \cite{clem-griff}.
The discriminant locus of the conic bundle 
is a curve $C \subset S$, with at most double points \cite[Prop. 1.2]{beauvillejaco}. The smooth points of $C$ correspond to two intersecting
lines, and the nodes to double lines. There is then a natural \'etale double cover
(an admissible cover if $C$ is singular \cite{beauvillejaco}) $\tilde{C} \to C$ of the curve $C$
associated to $X$. The intermediate
Jacobian $J(X)$ is then isomorphic to the Prym variety $P(\tilde{C}/C)$ as principally polarized Abelian
variety \cite{beauvillejaco}. This allows to show the non-rationality of conic bundles over $\PP^2$ with
discriminant curve of degree $\geq 6$ \cite{beauvillejaco}. Remark that if $S$ is not rational or $C$ disconnected,
then $X$ cannot be rational. We will then not consider these cases. Moreover, since $X$ is standard, $p_a(C)$ is
positive (see e.g. \cite[Sect. 1]{iskovconicduke}).

If $S$ is a minimal rational surface, then Shokurov \cite{shokuprym} has shown that $X$ is rational if
and only if $J(X)$ splits as the direct sum of Jacobians of smooth projective curves and that this happens only in
five cases: if $S=\PP^2$, either $C$ is a cubic, or a quartic, or $C$ is a quintic and $\tilde{C} \to C$ is given by
an even theta-characteristic; if $S= \FF_n$, either $C$ is hyperelliptic or $C$ is
trigonal and in both cases the map to $\PP^1$ is induced by the ruling of $S$. If $S$ is not minimal,
it is conjectured that there are essentially no more cases \cite{iskovconicduke}.

Our aim is to give a categorical approach to this problem, using semiorthogonal decompositions.
Indeed, in \cite{kuznetconicbundles} Kuznetsov considers 
the sheaf $\kb_0$ of even parts of Clifford algebras associated to the quadratic form defining the conic
fibration, and $\Db(S,\kb_0)$ the bounded derived category of coherent $\kb_0$-algebras over $S$.
He describes a fully faithful functor $\Phi: \Db(S,\kb_0) \to
\Db(X)$ and gives a semiorthogonal decomposition for the derived category of $X$
as follows:
$$\Db(X) = \langle \Phi\Db(S,\kb_0), \pi^* \Db(S) \rangle.$$
If $S$ is a rational surface, its derived category admits a full exceptional sequence, which leads
to the following semiorthogonal decomposition
\begin{equation}\label{decompDX}
\Db(X) = \langle \Phi\Db(S,\kb_0), E_1, \ldots, E_s \rangle,
\end{equation}
where $\{E_i\}_{i=1}^s$ are exceptional objects. The non-trivial information about the
geometry of the conic bundle is contained in the category $\Db(S, \kb_0)$. Note that in the case where $X$ is the blow-up
of a smooth cubic threefold $Y$ along a line, $\Db(S,\kb_0)$ contains $\cat{A}_Y$, which identifies the
isomorphism class of $Y$ \cite{noi-cubic}.
Remark that a different approach to the same problem, via generalized homological mirror symmetry, leads to the conjectures stated in
\cite{katzarkov1, katzarkov2}. Anyway we do not establish any link with the results described here.


\begin{theorem}\label{main-theorem1}
Let $\pi: X \to S$ be a standard conic bundle over a rational surface. Suppose that $\{\Gamma_i\}_{i=1}^k$
are smooth projective curves and $k\geq 0$, with fully faithful functors $\Psi_i: \Db(\Gamma_i) \to \Db(S,\kb_0)$
for $i=1,\ldots k$, such that $\Db(S,\kb_0)$ admits a semiorthogonal decomposition:
\begin{equation}\label{decompo-in-main-thm}
\Db(S,\kb_0) = \langle \Psi_1 \Db(\Gamma_1), \ldots, \Psi_k \Db(\Gamma_k), E_1, \ldots, E_l \rangle,
\end{equation}
where $E_i$ are exceptional objects and $l \geq 0$. Then $J(X) = \bigoplus_{i=1}^k J (\Gamma_i)$ as principally
polarized Abelian variety.
\end{theorem}

If $S$ is non-rational, and then so is $X$, Theorem \ref{main-theorem1} fails; its proof relies indeed strictly
on the rationality of $S$. In \ref{non-rat-ex} we provide an example of a standard conic bundle over a non-rational
surface with $\Db(S,\kb_0)$ decomposing in derived categories of smooth projective curves.

The interest of Theorem \ref{main-theorem1} is twofold:
first it is the first non-trivial example where informations on the birational properties
and on algebraically trivial cycles are obtained
directly from a semiorthogonal decomposition. Secondly it gives a categorical criterion of rationality for
conic bundles over minimal surfaces. Indeed, Shokurov \cite[10.1]{shokuprym} proves that
for such surfaces a conic bundle is rational if and only if the Jacobian splits.
We can also prove the other implication by a case by case analysis.

\begin{theorem}\label{main-theorem2}
If $S$ is a rational minimal surface, then $X$ is rational and $J(X) = \bigoplus_{i=1}^k J(\Gamma_i)$ if and only if
there are fully faithful functors $\Psi_i: \Db(\Gamma_i) \to \Db(S,\kb_0)$ and
a semiorthogonal decomposition
$$\Db(S,\kb_0) = \langle \Psi_1 \Db(\Gamma_1), \ldots, \Psi_k \Db(\Gamma_k), E_1, \ldots, E_l \rangle,$$
where $E_i$ are exceptional objects and $l\geq 0$.
\end{theorem}
 
The key of the proof of Theorem \ref{main-theorem1} is the study of the maps induced by a fully faithful functor
$\Psi: \Db(\Gamma) \to \Db(X)$ on the rational Chow motives, as explained in \cite{orlov-motiv}, where
$\Gamma$ is a smooth projective curve of positive genus.
In particular, the biggest step consists in proving that such a functor induces an injective morphism
$\psi: J(\Gamma) \to J(X)$ preserving the principal polarization. The existence of the
required semiorthogonal decomposition implies then the bijectivity of the sum of the $\psi_i$'s.


The paper is organized as follows: in Sections \ref{sec-prel} and \ref{sec-basic-on-conic}
we recall respectively basic facts about motives
and derived categories and the construction from \cite{orlov-motiv}, and
the description of motive, derived category and intermediate Jacobian of a conic bundle.
In Section \ref{sect-criterion} we prove Theorem \ref{main-theorem1}, and in Sections \ref{sect-plane} and
\ref{sect-hirze} we finish the
proof of Theorem \ref{main-theorem2}, analyzing respectively the case $S=\PP^2$ and
$S=\FF_n$.

\subsection*{Notations}

Except for Section \ref{sec-prel}, we work over the complex field $\CC$. Any triangulated
category is assumed to be essentially small. 
Given a smooth projective variety $X$, we denote by $\Db(X)$ the bounded derived 
category of coherent sheaves on it, by $K_0(X)$ its Grothendieck group, by $CH^d(X)$ the Chow group of
codimension $d$ cycles and by $A^d(X)$ the subgroup of algebraically trivial cycles in $CH^d(X)$. The subscript
$_{\QQ}$ is used there whenever we consider $\QQ$-coefficients, while $h(X)$ already denotes the
rational Chow motive.
We will denote $\Prym(\tilde{C}/C)$ the Prym motive and $P(\tilde{C}/C)$
the Prym variety for an admissible double cover $\tilde{C} \to C$. 
Whenever a functor between derived categories is given, it will be denoted as underived,
for example for $f: X \to Y$, $f^*$ and $f_*$ denote respectively the derived pull-back and push-forward.

\medskip

{\small\noindent{\bf Acknowledgements.}
This work has been developed during visits of the authors at the Humboldt University of Berlin,
the University of Duisburg--Essen, the Roma III University, and the University of Rennes that are warmly acknowledged.
The first named author is grateful to H.Esnault for pointing him out \cite{orlov-motiv}, and to
her and A.Chatzistamatiou for useful discussions. We thank A.Beauville for pointing out
a missing case in an early version. Moreover, it is a pleasure to thank the people that shared
their views with us and encouraged us during the writing of this paper. In alphabetical order: A.Beauville,
G.Casnati, I.Dolgachev, V.Kanev, L.Katzarkov, A.Kuznetsov, M.Mella, A.Verra.
The first named author was supported by the SFB/TR 45 `Periods, moduli spaces, and arithmetic of algebraic varieties'.}

\section{Preliminaries}\label{sec-prel}

In this Section, we recall some basic facts about motives, derived categories, semiorthogonal decompositions and
Fourier--Mukai functors. The
experienced reader can easily skip subsections \ref{intromotiv} and \ref{introderi}. In \ref{subs-FMandMotives},
we explain how a Fourier--Mukai functor induces a motivic map,
following \cite{orlov-motiv}, and we retrace the results from \cite{marcellocurves} under this point of view to give a baby
example clarifying some of the arguments we will use later.  

\subsection{Motives}\label{intromotiv}
We give a brief introduction to rational Chow motives, following \cite{scholl}. The most important results we will need
are the correspondence between the submotive $h^1(C) \subset h(C)$ of a smooth projective curve and its
Jacobian, and the Chow--K\"unneth decomposition of the motive of a smooth surface.

Let $X$ be a smooth projective scheme over a field $\Bbbk$. For any integer $d$, let $\kz^d(X)$ be the free Abelian
group generated by irreducible subvarieties of $X$ of codimension $d$. We denote by $CH^d(X) = \kz^d(S)/{\sim_{rat}}$ the
codimension $d$ Chow group and by $CH_{\QQ}^d(X) := CH^d(X) \otimes \QQ$. In this section, we are only concerned with
rational coefficients.
Let $Y$ be a smooth projective scheme. If $X$ is purely $d$-dimensional, we put, for any integer $r$,
$$\corr^r(X,Y) := CH_{\QQ}^{d+r}(X \times Y).$$
If $X = \coprod X_i$, where $X_i$ is connected, we put
$$\corr^r(X,Y) := \bigoplus \corr^r(X_i,Y) \subset CH_{\QQ}^*(X \times Y).$$
If $Z$ is a smooth projective scheme, the composition of correspondences is defined by a map
\begin{equation}\label{comp-of-corr}
\xymatrix@R=1pt{
\corr^r(X,Y) \otimes \corr^s(Y,Z) \ar[r] & \corr^{r+s} (X,Z) \\
f \otimes g \ar@{|->}[r] & g.f:=p_{13*}(p_{12}^* f . p_{23}^* g), 
}
\end{equation}
where $p_{ij}$ are the projections from $X \times Y \times Z$ onto products of two factors.

The category $\km_{\Bbbk}$ of Chow motives over $\Bbbk$ with rational coefficients is defined as follows: an object
of $\km_{\Bbbk}$ is a triple $(X,p,m)$, where $X$ is a variety, $m$ an integer and $p \in \corr^0(X,X)$ an idempotent,
called a \it projector\rm.
Morphisms from $(X,p,m)$ to $(Y,q,n)$ are given by elements of $\corr^{n-m}(X,Y)$ precomposed with $p$ and
composed with $q$:
\[
\Hom_{\km_{\Bbbk}}((X,p,m),(Y,q,n)) = q \corr^{n-m}(X,Y)p \subset \corr^{n-m}(X,Y).
\]

There exists a tensor product on $\km_{\Bbbk}$, that is defined on objects as follows:
$$(X,p,m)\otimes(Y,q,n)=(X\times Y,p\otimes q,m+n);$$
whereas on morphisms it is defined by
$$p_1f_1q_1\otimes p_2f_2q_2 =(p_1\otimes p_2)(f_1\otimes f_2)(q_1\otimes q_2) \in Corr^{n_1+n_2-m_1-m_2}(X_1\times X_2,Y_1\times Y_2)$$
if $p_if_iq_i: (X_i,p_i,m_i)\rightarrow (Y_i,q_i,n_i)$.
Moreover there is a natural functor $h$ from the category of smooth projective schemes to the category of motives,
defined by $h(X) = (X,\id,0)$, and, for any morphism $\phi: X \to Y$, $h(\phi)$ being the correspondence given
by the graph of $\phi$. A triple $(X,p,0)$ is then considered as a formal direct summand of
$h(X)$ corresponding to $p$.
We write $\QQ:=(\Spec {\Bbbk}, \id, 0)$ for the unit motive and $\QQ(-1) := (\Spec {\Bbbk}, \id, -1)$
for the Tate (or Lefschetz) motive,
and $\QQ(-i) := \QQ(-1)^{\otimes i}$ for $i > 0$. We denote $h(X)(-i) := h(X) \otimes \QQ(-i)$.
Given a triple $(X,p,0)$, the triple $(X,p,i)$ will then give a formal direct summand of $h(X)(i)$.
Finally, we have $\Hom(\QQ(-d),h(X))= CH_{\QQ}^d(X)$ for all smooth projective schemes $X$ and all integers $d$.

If $X$ is irreducible of dimension $d$ and has a rational point, the embedding $\alpha: pt \to X$ of the point defines a motivic map
$\QQ \to h(X)$. We denote by $h^0(X)$ its image: if $p_0 := pt \times X \subset X \times X$, the
triple $(X,p_0,0)$ gives $h^0(X)$. Denote by $h^{\geq 1}(X)$ the quotient of $h(X)$ via $h^0(X)$.
Similarly, we have that $\QQ(-d)$ is a quotient of $h(X)$, and we denote it by $h^{2d}(X)$:
if $p_{2d} := X \times pt \subset X \times X$, the
triple $(X,p_{2d},0)$ gives $h^{2d}(X)$.
For example, if $X = \PP^1$, we have that $h^{\geq 1}(\PP^1) = h^2(\PP^1)$ and then $h(\PP^1) \simeq \QQ \oplus \QQ(-1)$.
In the case of smooth projective curves of positive genus another factor which corresponds to the Jacobian variety
of the curve is appearing.

Let $C$ be a smooth projective connected curve with a rational point. Then one can define a motive $h^1(C)$ such that
we have a direct sum:
$$h(C) = h^0(C) \oplus h^1(C) \oplus h^2(C).$$
The main fact is that the theory of the motives $h^1(C)$ corresponds to that of Jacobian varieties (up to isogeny). Indeed we have
$$\Hom(h^1(C),h^1(C')) = \Hom(J(C),J(C'))\otimes \QQ.$$
In particular, the full subcategory of $\km_{\Bbbk}$ whose objects are direct summands of the motive $h^1(C)$ is equivalent
to the category of Abelian subvarieties of $J(C)$ up to isogeny.
Finally, for all $d$ there is no non-trivial map $h^1(C) \to h^1(C)$ factoring through $\QQ(-d)$. Indeed, we have
$\Hom(h^1(C),\QQ(-d)) = CH_{\QQ}^{1-d}(C)_{num=0}$, the numerically trivial part of the Chow group
of rational codimension $1-d$ cycles. Such group is zero unless $d=0$. On the other hand,
$\Hom(\QQ(-d),h^1(C))=CH_{\QQ}^d(C)_{num=0}$, the numerically trivial part of the Chow group
of rational codimension $d$ cycles. Such group is zero unless $d=1$.

Let $S$ be a surface. Murre constructed \cite{murre} the motives $h^i(S)$, defined by projectors $p_i$ in $CH_{\QQ}^i(S \times S)$
for $i=1,2,3$, and described a decomposition
$$h(S) = h^0(S) \oplus h^1(S) \oplus h^2(S) \oplus h^3(S) \oplus h^4(S).$$
We already remarked that $h^0(S) = \QQ$ and $h^4(S) = \QQ(-2)$. Roughly speaking, the submotive $h^1(S)$ carries the Picard variety,
the submotive $h^3(S)$ the Albanese variety and the submotive $h^2(S)$ carries the N\'eron--Severi group, the Albanese kernel
and the transcendental cycles. If $S$ is a smooth rational surface and $\Bbbk=\bar{\Bbbk}$,
then $h^1(S)$ and $h^3(S)$ are trivial, while
$h^2(S) \simeq \QQ(-1)^{\rho}$, where $\rho$ is the rank of the N\'eron--Severi group. In particular, the motive of
$S$ splits in a finite direct sum of (differently twisted) Tate motives.

In general, it is expected that if $X$ is a smooth projective $d$-dimensional variety, there exist projectors
$p_i$ in $CH_{\QQ}^i(X \times X)$ defining motives $h^i(X)$ such that $h(X) = \oplus_{i=0}^{2d} h^i(X)$. Such
a decomposition is called a \it Chow--K\"unneth decomposition\rm.
We have seen that the motive of any smooth projective curve or surface admits a Chow--K\"unneth decomposition.
This is true also for the motive of a smooth uniruled complex threefold \cite{angel-mstach}.

\subsection{Semiorthogonal decomposition, exceptional objects and mutations}\label{introderi}
We introduce here semiorthogonal decompositions, exceptional objects and mutations in a $\Bbbk$-linear triangulated category
$\cat{T}$,
following \cite{bondal-repr,bondal-kapranov,bondal-orlov}, and give some examples which will be useful later on.
Our only applications will be given in the case where $\cat{T}$ is the bounded derived category of a smooth
projective variety, but we stick to the more general context. 
Recall that whenever a functor between derived categories is given, it will be denoted as underived,
for example for $f: X \to Y$, $f^*$ and $f_*$ denote respectively the derived pull-back and push-forward.

A full triangulated category $\cat{A}$ of
$\cat{T}$ is called \it admissible \rm if the embedding functor admits a left and a right adjoint.

\begin{definition}[\cite{bondal-kapranov,bondal-orlov}]\label{def-semiortho}
A \it semiorthogonal decomposition \rm of $\cat{T}$ is a sequence of full admissible triangulated subcategories
$\cat{A}_1, \ldots, \cat{A}_n$ of $\cat{T}$ such that $\Hom_{\cat{T}}(A_i,A_j) = 0$ for all $i>j$ and for 
all objects $A_i$ in $\cat{A}_i$ and $A_j$ in $\cat{A}_j$, and for every object $T$ of $\cat{T}$, there is a chain
of morphisms $0=T_n \to T_{n-1} \to \ldots \to T_1 \to T_0 = T$ such that the cone of $T_k \to T_{k-1}$ is
an object of $\cat{A}_k$ for all $k=1,\ldots,n$. Such a decomposition will be written
$$\cat{T} = \langle \cat{A}_1, \ldots, \cat{A}_n \rangle.$$
\end{definition}

\begin{definition}[\cite{bondal-repr}]\label{def-except}
An object $E$ of $\cat{T}$ is called \it exceptional \rm if $\Hom_{\cat{T}} (E,E) = \Bbbk$, and $\Hom_{\cat{T}}(E,E[i])=0$
for all $i \neq 0$. A collection $(E_1,\ldots,E_l)$ of exceptional objects is called \it exceptional \rm if
$\Hom_{\cat{T}}(E_j,E_k[i])=0$ for all $j>k$ and for all integer $i$. If no confusion arises, given an
exceptional sequence $(E_1, \ldots, E_l)$ in $\cat{T}$, we will denote by $\cat{E}:= \langle E_1, \ldots, E_l \rangle$
the triangulated subcategory of $\cat{T}$ generated by the sequence. 
\end{definition}

If $E$ in $\cat{T}$ is an exceptional object, the triangulated category generated by $E$ (that is, the smallest full
triangulated subcategory of $\cat{T}$ containing $E$) is equivalent to the derived category of a point, seen as a smooth projective
variety. The equivalence $\Db(pt) \to \langle E \rangle \subset \cat{T}$ is indeed given by sending $\ko_{pt}$ to $E$.
Given an exceptional collection $(E_1,\ldots,E_l)$ in the derived category $\Db(X)$ of a smooth projective variety,
there is a semiorthogonal decomposition \cite{bondal-orlov}
$$\Db(X) = \langle \cat{A},\cat{E}\rangle,$$
where $\cat{A}$ is the full triangulated subcategory whose objects are all the $A$ satisfying $\Hom(E_i,A)=0$
for all $i=1,\ldots,l,$. We say that the exceptional sequence is \it full \rm if the category $\cat{A}$ is trivial.

There are many examples of smooth projective varieties admitting a full exceptional sequence.
For example the sequence $(\ko(i), \ldots, \ko(i+n))$ is full exceptional in $\Db(\PP^n)$ for all $i$ integer
\cite{beilinson}. If $X$ is an odd-dimensional smooth quadric hypersurface in $\PP^n$ and $\Sigma$ the spinor bundle,
the sequence $(\Sigma(i),\ko(i+1),\ldots,\ko(i+n))$ is full exceptional in $\Db(X)$ and a similar sequence
(with two spinor bundles) exists for even-dimensional smooth quadric hypersurfaces \cite{kapranovquadric}.

\begin{proposition}[\cite{orlovprojbund}]\label{prop-decom-projbund}
Let $X$ be a smooth projective variety and $F$ a locally free sheaf of rank $r+1$ over it.
Let $p: \PP(F) \to X$ be the associated projective bundle. The functor $p^*:\Db(X) \to 
\Db(\PP(F))$ is fully faithful and for all integer $i$ we have the semiorthogonal decomposition:
$$\Db(\PP(F)) = \langle p^* \Db(X) \otimes \ko_{\PP/X}(i), \ldots, p^* \Db(X) \otimes \ko_{\PP/X}(i+r) \rangle,$$
where $\ko_{\PP/X}(1)$ is the Grothendieck line bundle of the projectivization. 
\end{proposition}

\begin{proposition}[\cite{orlovprojbund}]\label{decomp-blow-up}
Let $X$ be a smooth projective variety, $Y \hookrightarrow X$ a smooth projective subvariety
of codimension $d>1$ and $\ee: \widetilde{X} \to X$ the blow-up of $X$ along $Y$. Let $D \stackrel{\iota}{\hookrightarrow}
\widetilde{X}$ be the exceptional divisor and $p: D \to Y$ the restriction of $\ee$. Then 
$\ee^*: \Db(X) \to \Db(\widetilde{X})$ is fully faithful and, for all $j$, there are
fully faithful functors $\Psi^j:\ee^* \Db(Y) \to D(\widetilde{X})$ giving
the following semiorthogonal decomposition:
$$\Db(\widetilde{X}) = \langle \Psi^{-d+1}\Db(Y), \ldots, \Psi^{-1} \Db(Y), \ee^* \Db(X) \rangle.$$
Notice that the fully faithful functors $\Psi^j$ are explicitly given by $\Psi^j(-) = \iota_* (p^*(-) \otimes \ko_{D/Y}(j))$.  
\end{proposition}
We will refer to Proposition \ref{decomp-blow-up} as the \it Orlov formula \rm for blow ups.

We finally remark that if $X$ has dimension at most 2 and is rational and $\Bbbk=\CC$, the derived
category $\Db(X)$ admits a full exceptional sequence. We have already seen this for $\PP^1$ and $\PP^2$. If $X$ is a
Hirzebruch surface, then it has a 4-objects full exceptional sequence by Prop. \ref{prop-decom-projbund} and the decomposition
of $\PP^1$. We conclude by the birational classification of 
smooth complex projective surfaces and the Orlov formula for blow-ups.
In particular a complex rational surface with Picard number $\rho$ has a full exceptional sequence of $\rho+2$ objects.

Given a semiorthogonal decomposition $\langle \cat{A}_1, \ldots \cat{A}_n \rangle$ of $\cat{T}$, we can define an operation
called \it mutation \rm (called originally, in Russian, \it perestroika\rm) 
which allows to give new semiorthogonal decompositions with equivalent components.
What we need here is the following fact, gathering
different results from \cite{bondal-repr}. 
\begin{proposition}
Suppose that $\cat{T}$ admits a semiorthogonal decomposition $\langle \cat{A}_1,\ldots, \cat{A}_n \rangle$.
Then for each $1 \leq k \leq n-1$, there exists an endofunctor is a semiorthogonal decomposition
$$\cat{T} = \langle \cat{A}_1, \ldots, \cat{A}_{k-1},L_{\cat{A}_k} (\cat{A}_{k+1}), \cat{A}_k, \cat{A}_{k+2},\ldots,\cat{A}_n \rangle,$$
where $L_{\cat{A}_k}: \cat{A}_{k+1} \to L_{\cat{A}_k} (\cat{A}_{k+1})$ is an equivalence, called the \rm left mutation \it
through $\cat{A}_k$. Similarly, for each $2 \leq k \leq n$, there is a semiorthogonal decomposition
$$\cat{T} = \langle \cat{A}_1, \ldots, \cat{A}_{k-2}, \cat{A}_k, R_{\cat{A}_k} (\cat{A}_{k-1}), \cat{A}_{k+1},\ldots,\cat{A}_n \rangle,$$
where $R_{\cat{A}_k}: \cat{A}_{k-1} \to R_{\cat{A}_k} (\cat{A}_{k-1})$ is an equivalence, called the \rm right mutation \it
through $\cat{A}_k$.
\end{proposition}

Remark in particular that the mutation of an exceptional object is an exceptional object. If $\cat{T}$ is the bounded
derived category of a smooth projective variety and $n=2$, there is a very useful explicit formula for left and right mutations.

\begin{lemma}[\cite{bondal-kapranov}]\label{lem-mut}
Let $X$ be a smooth projective variety and $\Db(X) = \langle \cat{A}, \cat{B} \rangle$ a semiorthogonal decomposition.
Then $L_{\cat{A}}(\cat{B}) = \cat{B} \otimes \omega_X$ and $R_{\cat{B}}({\cat{A}}) = \cat{A} \otimes \omega_X^{-1}$.
\end{lemma}

\subsection{Fourier--Mukai functors, motives and Chow groups}\label{subs-FMandMotives}

Fourier--Mukai functors are the main tool in studying derived categories of coherent sheaves. We recall
here the main properties of a Fourier--Mukai functor and how it interacts with other theories, such
as the Grothendieck group, Chow rings and motives. A more detailed treatment (except for motives, see
\cite{orlov-motiv}) can be found in \cite[Chap. 5]{huybrechts}.

Let $X$ and $Y$ be smooth projective varieties of dimension $n$ and $m$ respectively
and $\ke$ an object of $\Db(X \times Y)$. 
The \it Fourier--Mukai functor $\Phi_{\ke}: \Db(Y) \to \Db(X)$ with kernel \rm $\ke$ is given by
$\Phi_{\ke}(A) = p_*(q^* A \otimes \ke),$
where $p$ and $q$ denote the projections form $X \times Y$ onto $X$ and $Y$ respectively. We will sometimes 
drop the subscript $\ke$. If $Z$ is a smooth projective variety, $\Phi_{\ke} : \Db(Y) \to \Db(X)$ and
$\Phi_{\kf}: \Db(X) \to \Db(Z)$, then the composition $\Phi_{\kf} \circ \Phi_{\ke}$ is the Fourier--Mukai
functor with kernel 
\begin{equation}\label{compoformula}
\kg := p_{13*}(p_{12}^* \ke \otimes p_{23}^* \kf),
\end{equation}
where $p_{ij}$ are the projections from $Y \times X \times Z$ onto products of two factors. It is worth noting
the similarity between (\ref{compoformula}) and the composition of correspondences (\ref{comp-of-corr}).

A Fourier--Mukai functor $\Phi_{\ke}$ always admits a left and right adjoint which are the Fourier--Mukai
functors with kernel $\ke_L$ and $\ke_R$ resp., defined by
$$\begin{array}{lcr}
\ke_L := \ke^{\vee} \otimes p^* \omega_X [n] & {\text{and}} & \ke_R := \ke^{\vee} \otimes q^* \omega_Y [m].
\end{array}$$
A celebrated result from Orlov \cite{orlov-represent} shows that any fully faithful exact functor $F: \Db(Y) \to
\Db(X)$ with right and left adjoint is a Fourier--Mukai functor whose kernel is uniquely determined up to isomorphism.

Given the Fourier--Mukai functor
$\Phi_{\ke}: \Db(Y) \to \Db(X)$,
consider the element $[\ke]$ in $K_0(X \times Y)$, given by the alternate sum of the cohomologies of
$\ke$. Then we have a commutative diagram
\begin{equation}\label{comm-diag-db-k0}
\xymatrix{
\Db(Y) \ar[r]^{\Phi_{\ke}} \ar[d]_{[\phantom{a}]} & \Db(X) \ar[d]^{[\phantom{a}]} \\
K_0(Y) \ar[r]^{\Phi^K_{\ke}} & K_0(X),
}
\end{equation}
where $\Phi^K_{\ke}$ is the $K$-theoretical Fourier--Mukai transform defined by
$\Phi^K_{\ke}(A) = p_!(q^* A \otimes [\ke])$. If $\Phi_{\ke}$ is fully faithful,
we have $\Phi_{\ke_R} \circ \Phi_{\ke} =
\id_{\Db(Y)}$. This implies $\Phi^K_{\ke_R} \circ \Phi^K_{\ke} = \id_{K_0(Y)}$ and then $K_0(Y)$ is
a direct summand of $K_0(X)$.
\begin{lemma}\label{lemma-semiorth-and-K-groups}
Let $X$, $\{Y_i\}_{i=1,\ldots k}$ be smooth projective varieties, $\Phi_i: \Db(Y_i) \to \Db(X)$ fully
faithful functors and $\Db(X) = \langle \Phi_1 \Db(Y_1),\ldots,\Phi_k \Db(Y_k) \rangle$ a
semiorthogonal decomposition. Then $K_0 (X) = \bigoplus_{i=1}^k K_0 (Y_i)$, and there is
an isomorphism of $\QQ$-vector spaces
$CH_{\QQ}^*(X) \cong \bigoplus_{i=1}^k CH_{\QQ}^*(Y_i)$.
\end{lemma}
\begin{proof}
The full and faithful functors $\Phi_i : \Db(Y_i) \to \Db(X)$ have to be of Fourier--Mukai type
and then $K_0(Y_i)$ are direct summands of $K_0(X)$.
The generation follows from the definition of a semiorthogonal decomposition.
The isomorphism between Chow rings as vector spaces is a straightforward consequence of Grothendieck--Riemann--Roch
Theorem.
\end{proof}

Consider the element $e:= \ch([\ke]). \Td(X)$ in $CH_{\QQ}^*(X \times Y)$.
This gives a correspondence
$e: CH_{\QQ}^*(Y) \to CH_{\QQ}^*(X)$ and we have a commutative diagram

\begin{equation}\label{comm-diag-db-ch}
\xymatrix{
\Db(Y) \ar[r]^{\Phi_{\ke}} \ar[d] & \Db(X) \ar[d] \\
CH_{\QQ}^*(Y) \ar[r]^e & CH_{\QQ}^*(X),
}
\end{equation}
where the vertical arrows are obtained by taking the Chern character and multiplying with the Todd class.
The commutativity of the diagram follows from the Grothendieck--Riemann--Roch formula. Remark that here we used that
the relative Todd class of the projection $X \times Y \to X$ is $q^*\Td(Y)$.

As for the Grothendieck groups, the Chow ring and the rational cohomology (see \cite[Chapt. 5]{huybrechts}),
one can find a functorial correspondence between derived Fourier--Mukai functors and motivic maps. This
was first carried out by Orlov \cite{orlov-motiv}.
Indeed, the cycle $e$ is of mixed type in $CH_{\QQ}^*(X \times Y)$. Its components $e_i$ in $CH_{\QQ}^i(X \times Y)$ give motivic
maps $e_i: h(Y) \to h(X)(i-m)$. Denote by $\kf:=\ke_R$ the kernel of the right adjoint of $\Phi_{\ke}$, and 
$f = \ch([\kf]).\Td(Y)$  the associated cycle in $CH_{\QQ}^*(X \times Y)$. Then we get motivic maps $f_i: h(X) \to h(Y)(i-n)$,
that is $f_i: h(X)(n-i) \to h(Y)$. In particular, $f_{i}.e_{m+n-i}: h(Y) \to h(Y)$.
If we consider the cycles $e$ and $f$, the Grothendieck--Riemann--Roch formula
implies that $f.e= \oplus_{i=0}^{m+n} f_i .e_{m+n-i}$ induces the identity $\id: h(Y) \to h(Y)$.

\begin{example}\label{ex-fmofcurves}
As an example, we describe the result in \cite{marcellocurves} from the motivic point of view. This turns out
to be very useful in understanding the relationship between the derived category, the motive and
the Jacobian of a smooth projective curve, and contains some ideas that we will use in the proof
of Theorem \ref{main-theorem1}

Let $C_1$ and $C_2$ be smooth projective curves and $\Phi_{\ke}: \Db(C_1) \to \Db(C_2)$ a Fourier--Mukai functor.
In \cite{marcellocurves} it is shown that the map
$\phi: J(C_1) \to J(C_2)$ induced by $\Phi_{\ke}$ preserves the principal polarization if and only if $\Phi_{\ke}$ is
fully faithful, which is equivalent to ask that $\Phi_{\ke}$ is
an equivalence.

We could describe such result in the following way: consider the motivic maps $e_i: h(C_1) \to h(C_2)(i-1)$ where
$e$ is the cycle associated to $\ke$. We define $f$ as before via the right adjoint. If $\Phi_{\ke}$ is fully faithful,
then we have $f.e = \oplus_{i=0}^2 f_i. e_{2-i} = \id$.
Recall that a map $h^1(C) \to h^1(C)$ factoring through $\QQ(-d)$ is trivial for any integer $d$.
Moreover,  $h^0(C_j) = \QQ$, and $h^2(C_j) = \QQ(-1)$ for $j=1,2$. If we restrict $f.e$ to $h^1(C_1)$, we
get that $(f_i.e_{2-i})_{\vert h^1(C_1)} = 0$ unless
$i=1$, because for $i \neq 1$ this map factors through some $\QQ(-d)$.
In particular  \cite[2.2]{marcellocurves} we obtain that $(e_1.f_1)_{\vert h^1(C_1)}=\id_{h^1(C_1)}$ and then $h^1(C_1)$ is a
direct summand of $h^1(C_2)$. Applying the same argument to the adjoint
of $\Phi_{\ke}$, one obtains an isomorphism
$h^1(C_1) \simeq h^1(C_2)$. This gives an isogeny $J(C_1) \otimes \QQ \simeq J(C_2) \otimes \QQ$.

Moreover, the maps $e_1$ and $f_1$ are given both by $c_1([\ke])$, and they define a morphism $\phi: J(C_1) \to J(C_2)$
of Abelian varieties, with finite kernel. 
The key point to prove the preservation of the principal polarization is the fact that 
that the dual map $\hat{\phi}$ of $\phi$ is induced by
the adjoint of $\Phi_{\ke}$ (see \cite{marcellocurves}).
Being $\Phi_{\ke}$ a Fourier--Mukai functor carries indeed a deep amount of geometrical information.
\end{example}

\section{Derived categories, motives and Chow groups of conic bundles}\label{sec-basic-on-conic}

From now on, we only consider varieties defined over $\CC$.
Let $S$ be a smooth projective surface, and $\pi: X \to S$ a smooth standard conic bundle. By this, we mean
a flat projective surjective morphism whose 
scheme theoretic fibers are isomorphic to plane conics, such that for any irreducible curve $D \subset S$ the surface $\pi^{-1}(D)$ 
is irreducible (this second condition is also called relative minimality). The discriminant locus of the conic bundle 
is a curve $C \subset S$, which can be possibly empty, with at most double points \cite[Prop. 1.2]{beauvillejaco}.
The fiber of $\pi$ over a smooth point of $C$ is the union of two lines intersecting in a single point,
while the fiber over a node is a double line.
Recall that any conic bundle is birationally
equivalent to a standard one via elementary transformations \cite{sarkisov-bira}.

In this section, we recall known results about the geometry of $\pi: X \to S$. In section \ref{subs-deco} we 
deal with the decomposition of $h(X)$ described by Nagel and Saito \cite{nagel-saito} and with
the semiorthogonal decomposition of $\Db(X)$ described by Kuznetsov \cite{kuznetconicbundles}. In section \ref{subs-arno} we
recall the description of the intermediate Jacobian and the algebraically trivial part $A^2(X)$
of the Chow group given by \cite{beauvillejaco,beltrachow}. The order of the two sections reverses history, but
the decompositions of $h(X)$ and $\Db(X)$ hold in a more general framework.  

Before that, recall that to any standard conic bundle, we can associate an admissible double covering $\tilde{C} \to C$
of the curve $C$, that is $\tilde{C}$ has nodes exactly over the nodes of $C$ and the has degree 2 and
ramifies exactly over the nodes of $C$ (this is called a \it pseudo-rev\^etement \rm in \cite[D\'ef. 0.3.1]{beauvillejaco}.
The set of vertical lines of $X$ (that is, the ones
contained in a fiber) is then a $\PP^1$-bundle over $\tilde{C}$ \cite{beauvillejaco}. In the results recalled
here, if $C$ is not smooth, then it has to be replaced by its normalization and the corresponding double covering. Anyway, with no
risk of misunderstanding, we will tacitly assume this replacement when needed, and keep the notation $\tilde{C} \to C$.

\subsection{The decompositions of $h(X)$ and $\Db(X)$}\label{subs-deco}

Consider the rational Chow motive $h(X)$. Nagel and Saito \cite{nagel-saito} provide a relative Chow-K\"unneth
decomposition for $h(X)$. First of all, for a given double covering $\tilde{C} \to C$ of an irreducible curve with at
most double points, they define the \it Prym motive \rm $\Prym(\tilde{C}/C) := (\tilde{C},(1-\tau)/2,0)$
as a submotive of $h(\tilde{C})$, where $\tau$ is
the involution associated to the covering. In particular $\Prym^1(\tilde{C}/C)$ is
a submotive of $h^1(\tilde{C})$, $\Prym(\tilde{C}/C) = \Prym^1(\tilde{C}/C)$ if the double covering
is not trivial and $\Prym(\tilde{C}/C) = h(C)$ otherwise. We refrain here to give the details
of the construction, for which the reader can consult \cite{nagel-saito}. 
Moreover they show how $h(S)$ and $h(S)(-1)$ are direct summands of $h(X)$. Any conic
bundle (non necessarily standard) is uniruled and $h(X) = \oplus_{i=0}^6 h^i(X)$ is the Chow--K\"unneth decomposition
\cite{angel-mstach}. We have the following description:
$$
h^i(X) = h^i(S) \oplus h^{i-2}(S)(-1) \oplus \bigoplus_{j=1}^r\Prym^{i-2} (\tilde{C_j}/C_j)(-1),
$$
where $C_j$, for $j=1, \ldots r$, are the irreducible components of the discriminant curve $C$ \cite{nagel-saito}.

If $\pi:X \to S$ is standard, then there is no component of $C$ over which the double cover is trivial. It follows that
$h^i(X) = h^i(S) \oplus h^{i-2}(S)(-1)$ for $i \neq 3$ and 
$$h^3(X) = h^3(S) \oplus h^1(S)(-1) \oplus \bigoplus_{j=1}^r \Prym^1(\tilde{C_j}/C_j)(-1).$$

We will focus on the case where $S$ is a rational surface and $C$ is connected (in any other case, the
conic bundle is not rational). We finally end up, recalling section \ref{intromotiv}, with:
\begin{equation}\label{motivic-decomp}
\begin{array}{rl}
h^i(X) &= h^i(S) \oplus h^{i-2}(S)(-1) \,\,\,\,\,\, \text{if} \,\, i \neq 3,\\
h^3(X) &= \Prym^1(\tilde{C}/C)(-1),
\end{array}
\end{equation}
and in particular, for $i \neq 3$, $h^i(X)$ is either trivial or a finite sum of Tate motives.
\medskip
\begin{table}[h]
\centering
\begin{tabular}{c|c|c|c|c|c|c}

\noalign{\smallskip}

$h^0(X)$ & $h^1(X)$ & $h^2(X)$ & $h^3(X)$ & $h^4(X)$ & $h^5(X)$ & $h^6(X)$ \\\hline

$\QQ$ & $0$ & $\QQ^{\rho+1}(-1)$ & $\Prym^1(\tilde{C}/C)(-1)$ & $\QQ^{\rho+1}(-2)$ & $0$ & $\QQ(-3)$\\   

\end{tabular}
\medskip\caption{\label{table}\small The motive of a standard conic bundle $X$ with discriminant double cover $\tilde{C} \to C$
over a rational surface $S$ of Picard number $\rho$}
\end{table}
 
Consider the derived category $\Db(X)$. The fibers
of $\pi$ are plane conics and that there is a locally free rank 3 vector bundle $\ke$ on $S$,
such that $X \subset \PP(\ke)$ is the zero locus of a section $s: \ko_S(-1) \to Sym^2(\ke^*)$ and the
map $\pi$ is the restriction of the fibration $\PP(\ke) \to S$ \cite{beauvillejaco,kuznetconicbundles}.
Kuznetsov defines, in the more general framework of
any quadric fibration over any smooth projective manifold, the sheaves of even (resp. odd) parts $\kb_0$ (resp. $\kb_1$)
of the Clifford algebra $\kb$ associated to the section $s$.
One can consider the Abelian category $\coh(S,\kb_0)$ of
coherent sheaves with a structure of $\kb_0$-algebra and its bounded derived category
$\Db(S,\kb_0)$. In the case of a conic bundle, both $\kb_0 = \ko_S \oplus (\Lambda^2(\ke) \otimes \ko_S(-1))$,
and $\kb_1 = \ke \oplus \det(\ke) (-1)$ are locally free sheaves of rank 4. 

\begin{proposition}[\cite{kuznetconicbundles}]\label{decom-conic-bd}
Let $\pi: X \to S$ be a conic bundle and $\kb_0$ the sheaf of even parts of the Clifford algebra associated to it.
Then $\pi^*: \Db(S) \to \Db(X)$ is fully faithful and there is a fully faithful functor $\Phi: \Db(S,\kb_0) \to
\Db(X)$ such that
$$\Db(X) = \langle \Phi \Db(S,\kb_0), \pi^* \Db(S) \rangle.$$
\end{proposition}
We will refer to Proposition \ref{decom-conic-bd} as the \it Kuznetsov formula \rm for conic bundles.
Remark that Kuznetsov actually gives a similar semiorthogonal decomposition for any quadric fibration over any smooth
projective manifold. If in particular $S$ is a smooth complex rational surface with Picard number $\rho$,
its derived category admits a full exceptional sequence. It follows that

$$\Db(X) = \langle \Phi \Db(S,\kb_0), \cat{E} \rangle,$$

where $\cat{E}=<E_1, \ldots, E_{\rho+2}>$, the the pull back of the full exceptional sequence of $\Db(S)$.

We conclude this section by showing that, if $S$ is rational, the Clifford algebra of a standard conic bundle, and hence
the dervied category $\Db(S,\kb_0)$, are completely determined by the discriminant double cover. This will be very useful
in Sections \ref{sect-plane} and \ref{sect-hirze}, where we will complete the proof of Theorem \ref{main-theorem2}
by giving an example of each possible discriminant double cover. 

\begin{lemma}\label{lem-square-equiv}
Let $S$ be a smooth rational simply connected surface and $\pi: X \to S$, and $\pi':X'\to S$ be standard conic bundles with
associated sheaves of even parts of Clifford algebras $\kb_0$ and $\kb_0'$ respectively. If $X$ and
$X'$ have the same discriminant double cover $\tilde{C} \to C$, then $\kb_0$ is isomorphic to $\kb_0'$.
In particular, $\Db(S,\kb_0)$ and $\Db(S,\kb_0')$ are equivalent.
\end{lemma}
\begin{proof}
Let we denote by $K(S)$ the function field of $S$, and by $\Br(-)$ the
Brauer group. A quaternion algebra $\ka_{\eta}$ is an element of order 2 of $\Br(K(S))$. There is
an exact sequence \cite{artin-mumford}:
$$0 \longrightarrow \Br(S) \longrightarrow \Br(K(S)) \stackrel{\alpha}{\longrightarrow}
\bigoplus_{D \subset S} H^1_{et} (D, \QQ/\ZZ) \stackrel{\beta}{\longrightarrow} \bigoplus_{x \in S} \mu^{-1},$$
where in the third (resp. fourth) term the sum runs over curves $D$ contained in (resp. points $x$ of) $S$.
Recall that all elements of order two in $\Br(K(S))$ are quaternion algebras \cite{merkurev}.

In particular, the algebra $\kb_0$ defines over $K(S)$ a quaternion algebra, determined
up to an element of $\Br(S)$.
If $S$ is rational and simply connected,
then $\Br(S)=0$, and the map $\alpha$ is injective. In this case, we have a 1-1 correspondence
between quaternion algebras $\ka_{\eta}$ and standard conic bundles, as explained
in \cite{sarkisov-struct, iskovconicduke}.
\end{proof}
\begin{remark}
A similar argument was first developed by Panin (\cite{panin} page 450-51)
in the case of conic bundles on $\PP^2$ with a quintic discriminant curve. 
\end{remark}

\subsection{Algebraically trivial cycles on $X$ and Prym varieties}\label{subs-arno}

Given a curve $C$ with at most double points and an admissible double covering $\nu:\tilde{C} \to C$ one
can define the \it Prym variety \rm $P(\tilde{C}/C)$ as the connected component containing $0$ of
the kernel of the norm map $Nm: J(\tilde{C}) \to  J(C)$, sending a degree 0 divisor $D$ on $\tilde{C}$ to the degree 0 divisor $\nu_*D$ on $C$. Remark that here we are abusing of notations in the case where $C$ (and hence $\tilde{C}$) are singular:
in this case we have to replace them with their normalizations and the induced double cover, which
we denote still by $\tilde{C} \to C$ by abuse of notations.
The Prym variety is a principally polarized Abelian subvariety of $J(\tilde{C})$ of index 2 (\cite{mumfordprym, beauvilleprym}).

Let $\pi: X \to S$ be a standard conic bundle with associated double covering $\tilde{C} \to C$.
If $S=\PP^2$, Beauville showed that the intermediate Jacobian $J(X)$ is
isomorphic as a principally polarized Abelian variety to $P(\tilde{C}/C)$ \cite{beauvillejaco}.
Moreover, he shows that the algebraically trivial part $A^2(X)$ of $CH^2(X)$ 
is isomorphic to the Prym variety $P(\tilde{C}/C)$. The key geometric point is that the family of vertical lines
(that is, lines contained in a fiber of $\pi$) in $X$ is a $\PP^1$-bundle over the curve $\tilde{C}$.
There is then a surjective morphism $g: J(\tilde{C}) \to A^2(X)$ extending the map associating to a point
$c$ of $\tilde{C}$ the line $l_c$ over it. The isomorphism $\xi : P(\tilde{C}/C) \to A^2(X)$ is obtained by 
taking the quotient via $\ker(g)$. The inverse isomorphism $G=\xi^{-1}$ is a regular map making of $P(\tilde{C}/C)$ the
algebraic representative of $A^2(X)$ (for more details, see \cite[Ch. III]{beauvillejaco}).
Similar techniques prove the same results for any $S$ rational \cite{beltra-francia,beltrachow}.

\begin{definition}[\cite{beauvillejaco}, D\'ef 3.4.2]\label{def-inc-pol}
Let $Y$ be a smooth projective variety of odd dimension $2n+1$ and $A$ (an Abelian
variety) the algebraic representative of $A^{n+1}(Y)$ via the canonical map
$G:A^{n+1}(Y) \to A$. A polarization of $A$ with class $\theta_A$
in $\corr(A,A)$, is the \it incidence polarization \rm with respect to $Y$ if
for all algebraic maps $f: T \to A^{n+1}(Y)$ defined by a cycle $z$
in $CH^{n+1}(Y \times T)$, we have
$$(G \circ f)^* \theta_A = (-1)^{n+1} I(z),$$
where $I(z)=z.z$ in $\corr(T,T)$ is the composition of the correspondences $z \in \corr(T,Y)$ and $z \in \corr(Y,T)$.
\end{definition}

\begin{proposition}\label{prop-incid-polar}
Let $\pi: X \to S$ be a standard conic bundle over a smooth rational surface.
The principal polarization $\Theta_P$ of $P(\tilde{C}/C)$ is the incidence polarization with respect to $X$.
\end{proposition}
\begin{proof}
We prove the statement in the case where $C$ is smooth. In the case of nodal curves, one has to go through the
normalization, and this is just rewriting the proof of \cite[Thm. 3.6, (iii)]{beauvillejaco}. 

If $S=\PP^2$, this is \cite[Prop. 3.5]{beauvillejaco}. If $S$ is not $\PP^2$, consider the isomorphism
$\xi$. The proof of \cite[Prop. 3.3]{beauvillejaco}
can be rephrased in this setting, in particular, recalling the diagram in \cite[Pag. 83]{beltra-francia}, one
can check that the map $2 \xi$ is described by a cycle $y$ in $CH^2(X \times P(\tilde{C}/C))$.
Let $f: T \to A^2(X)$ be an algebraic map defined by a cycle $z$ in $CH^2(X \times T)$. Denoting
by $u := G \circ f$ and $u':= (\id_X,u)$, the map $2f$ is defined by the cycle $(u')^* y$.
The proof is now the same as the one of \cite[Prop. 3.5]{beauvillejaco}.
\end{proof}

\section{Reconstructing the intermediate Jacobian}\label{sect-criterion}

The first main result of this paper is the reconstruction of $J(X)$ as the direct
sum of Jacobians of smooth projective curves, starting from a semiorthogonal decomposition
of $\Db(S,\kb_0)$. This Section is entirely dedicated to the proof of Theorem \ref{main-theorem1}.

\begin{blank}
Let $\pi: X \to S$ be a standard conic bundle over a rational surface. Suppose that $\{\Gamma_i\}_{i=1}^k$
are smooth projective curves and $k\geq 0$, with fully faithful functors $\Psi_i: \Db(\Gamma_i) \to \Db(S,\kb_0)$
for $i=1,\ldots k$, such that $\Db(S,\kb_0)$ admits an exceptional sequence $(E_1, \ldots, E_l)$ and a semiorthogonal decomposition:
\begin{equation}\label{decompo-in-main-thm2}
\Db(S,\kb_0) = \langle \Psi_1 \Db(\Gamma_1), \ldots, \Psi_k \Db(\Gamma_k), \cat{E} \rangle.
\end{equation}
Then $J(X) = \bigoplus_{i=1}^k J (\Gamma_i)$ as principally polarized Abelian variety.
\end{blank}

If $S$ is minimal, we obtain the ``if'' part of Theorem \ref{main-theorem2} 
combining the reconstruction of the Jacobian of Theorem \ref{main-theorem1}
with Shokurov's rationality criterion \cite[Thm. 10.1]{shokuprym}.

\begin{corollary}\label{ratio-criterion}
If $\pi: X \to S$ is a standard conic bundle over a minimal rational surface and
$$\Db(S,\kb_0) = \langle \Psi_1 \Db(\Gamma_1), \ldots, \Psi_k \Db(\Gamma_k), \cat{E} \rangle,$$
where $\Gamma_i$ are smooth projective curves, $\Psi_i: \Db(\Gamma_i) \to \Db(S,\kb_0)$ are full and faithful functors, and
$\cat{E}$ is generated by exceptional objects, then $X$ is rational and $J(X) = \bigoplus_{i=1}^k J (\Gamma_i)$.
\end{corollary}

If we have the decomposition (\ref{decompo-in-main-thm2}), using Prop. \ref{decom-conic-bd} and
that $S$ is rational of Picard number $\rho$, we get
\begin{equation}\label{decomp-tot}
\Db(X) = \langle \Psi_1 \Db(\Gamma_1), \ldots, \Psi_k \Db(\Gamma_k), E_1, \ldots, E_r \rangle,
\end{equation}
where $E_i$ are exceptional objects, $r=l+\rho+2>0$, and we denote by $\Psi_i$, by abuse of notation, the
composition of the full and faithful functor $\Psi_i$ with the full and faithful functor $\Db(S,\kb_0) \to \Db(X)$.
Remark that we can suppose that $\Gamma_i$ has positive genus for all $i=1,\ldots,k$. Indeed, the derived
category of the projective line admits a semiorthogonal decomposition by two exceptional objects. Then
if there exists an $i$ such that $\Gamma_i \simeq \PP^1$, it is enough to perform some mutation to get
a semiorthogonal decomposition like (\ref{decomp-tot}) with $g(\Gamma_i) > 0$ for all $i$ (recall we do not exclude the
case $k=0$).
By Lemma \ref{lemma-semiorth-and-K-groups} we have an isomorphism of $\QQ$-vector spaces:
\begin{equation}\label{orth-decomp-chow}
CH_{\QQ}^*(X) = \bigoplus_{i=1}^k CH_{\QQ}^*(\Gamma_i) \oplus \QQ^r,
\end{equation}
where we used the fact that the category generated by a single exceptional object is equivalent to the derived category
of a point, and $CH_{\QQ}^*(pt) = \QQ$. We are interested in understanding how the decomposition (\ref{orth-decomp-chow})
projects onto the codimension 2 cycle group $CH_{\QQ}^2(X)$ and in particular onto the algebraically
trivial part.

The proof is in two parts: first if $\Psi: \Db(\Gamma) \to \Db(X)$ is fully faithful and $\Gamma$
has positive genus, we get that $J(\Gamma)$ is isomorphic to a principally polarized Abelian
subvariety of $J(X)$ (Prop. \ref{prop-interm}). This is essentially based on constructions from
\cite{orlov-motiv} and results from \cite{beauvillejaco}. In the second part, starting
from the semiorthogonal decomposition we deduce the required isomorphism.

\begin{lemma}\label{injective-isogeny}
Let $\Gamma$ be a smooth projective curve of positive genus. Suppose there is a fully faithful functor
$\Psi: \Db(\Gamma) \to \Db(X)$. Then $J(\Gamma)$ is isogenous to an Abelian subvariety of $J(X) \simeq P(\tilde{C}/C)$.
\end{lemma}
\begin{proof}
Let $\ke$ be the kernel of the fully faithful functor $\Psi: \Db(\Gamma) \to \Db(X)$,
and $\kf$ the kernel of its right adjoint.
If we consider the cycles $e$ and $f$ described in Section \ref{subs-FMandMotives}, the Grothendieck--Riemann--Roch formula
implies that $f.e$ induces the identity $\id: h(\Gamma) \to h(\Gamma)$. If $e_i$ and $f_i$ are the $i$-th codimensional
components of $e$ resp. of $f$ in $CH_{\QQ}^*(X \times \Gamma)$, then $f.e = \oplus f_i . e_{4-i}$.
Remark that $e_i$ gives a map $h(\Gamma) \to h(X) (i-1)$. 
Recall that the motive $h(X)$ is decomposed, as in Table \ref{table}, by Tate motives and the Prym motive.
If we restrict to $h^1(\Gamma)$, recalling that a map $h^1(\Gamma) \to h^1(\Gamma)$ factoring through
a Tate motive is trivial, we get $(f_i.e_{4-i})_{\vert h^1(\Gamma)} = 0$ for all $i \neq 2$.
This implies that $\id_{h^1(\Gamma)} = (f_2 . e_2)_{\vert h^1(\Gamma)}$, and then that
$h^1(\Gamma)$ is a direct summand of $h(X)(1)$ and in particular it is a direct summand of $\Prym^1(\tilde{C}/C)$, which
proves the claim.
\end{proof}

Remark that we can describe explicitly the map $\psi_{\QQ}: J_{\QQ}(\Gamma) \to J_{\QQ}(X)$ induced by $\Psi$,
following the ideas in \cite{marcellocurves}. Indeed the map $\psi_{\QQ}$
is given by $e_2$, the codimension 2 component of the cycle associated to the kernel $\ke$.
Then $\psi_{\QQ}$ can be calculated just applying the Grothendieck--Riemann--Roch Theorem.

Let $p: \Gamma \times X \to
X$ and $q: \Gamma \times X \to \Gamma$ be the two projections. For $M$ in $J(\Gamma)$ (which we identify with
$\Pic^0(\Gamma)$) we calculate
the second Chern character $(\ch(\Psi(M))_2$ to get an element of $J(X)$ (which we identify with $A^2(X)$).
Applying Grothendieck--Riemann--Roch and using multiplicativity of Chern characters, we have the following:
$$(\ch(p_*(q^*M \otimes \ke)))_2 = p_*(\ch(q^*M).\ch(e).(1-(1/2) q^* K_{\Gamma}))_3,$$
since the relative dimension of $p$ is 1 and the relative Todd class is $1-(1/2) q^* K_{\Gamma}$.
Recalling that $\ch(q^*M) = 1 + q^*M$ and $q^*M . q^* K_{\Gamma} = 0$, we get
$$(\ch(p_*(q^*M \otimes \ke)))_2 = p_*(q^*M. \ch_2(\ke) - (1/2) q^*K_{\Gamma}.\ch_2(\ke) + \ch_3(\ke)).$$
It is clear that this formula just defines an affine map $\Psi^{CH}: CH_{\QQ}^1 (\Gamma) \to CH_{\QQ}^2(X)$ of $\QQ$-vector
spaces. In order to get the isogeny $\psi_{\QQ}$, we have to linearize and restrict to $J(\Gamma) \otimes \QQ$, to get finally:
$$\begin{array}{rl}
\psi_{\QQ} : J(\Gamma) \otimes \QQ &\longrightarrow J(X) \otimes \QQ \\
M & \mapsto p_*(q^*M. \ch_2(\ke))   
  \end{array}
$$
Now that we have the cycle describing the map $\psi_{\QQ}$, we obtain a unique morphism $\psi: J(\Gamma) \to J(X)$,
whose kernel can only be torsion. That is, we have an isogeny $\psi$ between $J(\Gamma)$ and an Abelian subvariety of
$J(X)$.

\begin{remark}\label{functoriality-and-variation}
Arguing as in \cite[Sect. 2.3]{marcellocurves}, we can show that the correspondence between $\Psi$ and $\psi$ is
functorial. Moreover, the functor with kernel $\ke[n]$ induces the map $(-1)^n\psi$. The functor with kernel
$\ke^{\vee}$ induces the map $\psi$. Given line bundles $L$ and $L'$ on $\Gamma$ and $X$ respectively,
the functor with kernel $\ke \otimes p^*L \otimes q^*L'$ induces the map $\psi$. The adjoint functor
of $\Psi$ is a Fourier--Mukai functor whose kernel is $\ke^{\vee} \otimes q^* \omega_X [3]$. Its composition
with $\Psi$ gives the identity of $\Db(\Gamma)$. The motivic map $f_2: \Prym^1(\tilde{C}/C)(-1) \to h^1(\Gamma)$ is
then given by the cycle $-\ch_2(e)$.
Then, by functoriality and (\ref{compoformula}), the cycle $I(\ch_2(e))$, as defined in Def. \ref{def-inc-pol},
is $-\id$ in $\corr(J(\Gamma),J(\Gamma))$. 
\end{remark}

Recall that, by \cite[Sec. 3]{beauvillejaco}, \cite{beltra-francia,beltrachow} and Proposition \ref{prop-incid-polar},
$P(\tilde{C}/C)$ is the algebraic representative of $A^2(X)$ and the principal polarization
$\Theta_{P}$ of $P(\tilde{C}/C)$ is the incidence polarization with respect to $X$.
In particular, we have an isomorphism $\xi : P(\tilde{C}/C) \to A^2(X)$ whose inverse $G$ makes the Prym
variety the algebraic representative of $A^2(X)$. Moreover, if $f:T \to A^2(X)$ is an algebraic map
defined by a cycle $z$ in $CH_{\QQ}^2(X \times T)$, then, according to Definition \ref{def-inc-pol}, we have
$$
(G \circ f)^* \theta_P = I(z).
$$
The map $\psi$ is defined by the cycle $\ch_2(e)$ in $CH_{\QQ}^2(X \times \Gamma)$.
Following Remark \ref{functoriality-and-variation}, the cycle $I(\ch_2(e))$ in $CH_{\QQ}^1(\Gamma \times \Gamma)$
gives the correspondence $-\id$, that is
$$(G \circ \psi)^* \theta_P = -\id.$$ 
Now going through the proof of \cite[Prop. 3.3]{beauvillejaco}, it is clear that
$$\psi^* \theta_{J(X)} = \id,$$
where $\theta_{J(X)}$ is the class of principal polarization of $J(X)$. Hence we get an injective morphism $\psi: J(\Gamma) \to
J(X)$ preserving the principal polarization. We can state the following result.

\begin{proposition}\label{prop-interm}
Let $\pi: X \to S$ be a standard conic bundle over a rational surface. Suppose that there is a smooth projective curve
$\Gamma$ of positive genus and a fully faithful functor $\Psi: \Db(\Gamma) \to \Db(S,\kb_0)$. Then there is an
injective morphism $\psi: J(\Gamma) \to J(X)$ of Abelian varieties, preserving the principal polarization.
\end{proposition}

Consider the projection $pr: CH_{\QQ}^*(X) \to CH_{\QQ}^2(X)$. The decomposition (\ref{orth-decomp-chow})
as a $\QQ$-vector space is rewritten as:
$$CH_{\QQ}^*(X) = \bigoplus_{i=0}^k \Pic_{\QQ}(\Gamma_i) \oplus \QQ^{r+k},$$
where we used that $CH_{\QQ}^*(\Gamma_i) = \Pic_{\QQ} (\Gamma_i) \oplus \QQ$.
This decomposition
is induced by the Fourier--Mukai functors $\Phi_i: \Db(\Gamma_i) \to \Db(X)$, then
the previous arguments show that
$pr$ restricted to $\oplus_{i=1}^k \Pic_{\QQ}^0(\Gamma_i)$ is injective and has image in $A_{\QQ}^2(X)$.
This map corresponds on each direct summand to the injective map $\psi_{i,\QQ}$ obtained as in Lemma \ref{injective-isogeny}.
Then the restriction of $pr$ to $\oplus_{i=1}^k \Pic_{\QQ}^0(\Gamma_i)$ corresponds to the sum of all those maps,
and we denote it by $\psi_{\QQ}$. Consider now the diagram
$$\xymatrix{
0 \ar[r] & \bigoplus_{i=0}^k \Pic_{\QQ}^0(\Gamma_i) \ar[r] \ar@{_{(}->}[d]_{pr=\psi_{\QQ}} & 
\bigoplus_{i=0}^k \Pic_{\QQ} (\Gamma) \oplus \QQ^{k+r} \ar[d]_{pr} \ar[dr]^{\bar{pr}} \\
0 \ar[r] & A_{\QQ}^2(X) \ar[r] & CH_{\QQ}^2(X) \ar[r] & CH_{\QQ}^2(X) / A_{\QQ}^2(X) \ar[r] & 0, 
}$$ 
where $\bar{pr}$ denotes the composition of $pr$ with the projection onto the the quotient. 

Denote by $J:=\psi(\oplus \Pic^0(\Gamma_i))$ the image of $\psi$ and $J_{\QQ}:=J \otimes \QQ$. 
By the above diagram, we have that the cokernel $A_{\QQ}^2(X)/J_{\QQ}$ is a finite dimensional $\QQ$-vector space,
since it has to be contained in $(\oplus (\Pic_{\QQ}(\Gamma_i)/\Pic_{\QQ}^0(\Gamma_i)) \oplus \QQ^{k+r}$.
Since $\psi$ is a morphism of Abelian varieties, its cokernel is also an Abelian variety, and then
it has to be trivial. This gives the surjectivity of $\psi$ and proves Theorem \ref{main-theorem1}.

\begin{remark}\label{numb-of-obj}
Let $\rho$ be the rank of the Picard group of $S$.
The numbers $k$ and $l$ of curves of positive genus and exceptional object respectively
in the semiorthogonal decomposition of Theorem \ref{main-theorem1}  satisfy a linear equation:
using the decomposition (\ref{orth-decomp-chow}), we obtain $l=2+\rho-2k$.
\end{remark}

\section{Rational conic bundles over the plane}\label{sect-plane}

Let $\pi:X \to \PP^2$ be a rational standard conic bundle. In particular, this implies that $C$ has positive
arithmetic genus (see e.g. \cite[Sect. 1]{iskovconicduke}). There are only three non-trivial possibilities for the
discriminant curve (\cite{beauvillejaco, shokuprym, iskocongru}). In fact, $X$ is rational if and only if 
is a quintic and the double covering $\tilde{C} \to C$ is given by an even theta characteristic, or $C$ is a quartic 
or a cubic curve.
As proved in Lemma \ref{lem-square-equiv}, once we fix the discriminant curve and the associated double
cover, we fix the Clifford algebra $\kb_0$. We then construct for
any such plane curve and associated double cover a model of rational standard conic bundle $X$ for which we provide the required
semiorthogonal decomposition. We analyze the three cases separately.

\subsection{Degree five discriminant}

Suppose $C$ is a degree 5 curve and $\tilde{C} \to C$ is given by 
an even theta-characteristic. Recall the description of a birational map $\chi: X \to \PP^3$
from \cite{panin} (see also \cite{iskovconicduke}).
There is a smooth curve $\Gamma$ of genus 5 and degree 7 in $\PP^3$ such that
$\chi: X \to \PP^3$ is the blow-up of $\PP^3$ along $\Gamma$. In fact the conic bundle $X\to \PP^2$
is obtained \cite{iskovconicduke} by resolving the linear system of cubics in $\PP^3$ vanishing on $\Gamma$.
Remark that $J(X)$ is isomorphic to $J(\Gamma)$ as a principally polarized Abelian variety.
 
Let $\pi: X \to \PP^2$ be the conic bundle
structure. We denote by $\ko(h)$ the pull back of $\ko_{\PP^2}(1)$ via $\pi$. Let us denote by $\ko(H)$ the pull-back
of $\ko_{\PP^3}(1)$ via $\chi$, and by $D$ the exceptional divisor. The construction of the
map $\pi$ gives $\ko(h)=\ko(3H-D)$, then we have $\ko(D) = \ko(3H-h)$. The canonical bundle $\omega_X$ is given
by $\ko(-4H+D)= \ko(-H-h)$.

\begin{proposition}\label{prop:2ndcase}
Let $\pi:X \to \PP^2$ be a standard conic bundle whose discriminant curve $C$ is a degree 5 curve and $\tilde{C} \to C$ is given 
by an even theta-characteristic.
Then there exists an exceptional object $E$ in $\Db(\PP^2,\kb_0)$ such that (up to equivalences):
$$\Db(\PP^2,\kb_0) = \langle \Db(\Gamma), E\rangle,$$
where $\Gamma$ is a smooth projective curve such that $J(X) \simeq J(\Gamma)$ as a principally polarized Abelian variety.
\end{proposition}

\begin{proof}
Consider the blow-up $\chi: X \to \PP^3$. Orlov formula (see Prop. \ref{decomp-blow-up}) provides
a fully faithful functor $\Psi: \Db(\Gamma) \to \Db(X)$ and a semiorthogonal decomposition:
$$
\Db(X) = \langle \Psi \Db(\Gamma), \chi^* \Db(\PP^3) \rangle.
$$
The derived category $\Db(\PP^3)$ has a full exceptional sequence
$\langle\ko_{\PP^3}(-2),\ko_{\PP^3}(-1),\ko_{\PP^3}, \ko_{\PP^3}(1) \rangle$. We get then the semiorthogonal decomposition:
\begin{equation}\label{decomp-blow-quintic}
\Db(X) = \langle \Psi \Db(\Gamma), \ko(-2H), \ko(-H), \ko, \ko(H) \rangle.
\end{equation}
Kuznetsov formula (see Prop. \ref{decom-conic-bd}) provides the decomposition:
$$
\Db(X) = \langle \Phi \Db(\PP^2,\kb_0), \pi^* \Db(\PP^2) \rangle.
$$
The derived category $\Db(\PP^2)$ has a full exceptional sequence
$\langle\ko_{\PP^2}(-1),\ko_{\PP^2}, \ko_{\PP^2}(1) \rangle$. We get then the semiorthogonal decomposition
\begin{equation}\label{decomp-conic-quintic}
\Db(X) = \langle \Phi \Db(\PP^2,\kb_0), \ko(-h), \ko, \ko(h) \rangle.
\end{equation}
We perform now some mutation to compare the decompositions \ref{decomp-blow-quintic} and \ref{decomp-conic-quintic}.

Consider the decomposition \ref{decomp-conic-quintic} and mutate $\Phi \Db(\PP^2,\kb_0)$ to the right through $\ko(-h)$.
The functor $\Phi' = \Phi \circ R_{\ko(-h)}$ is full and faithful and we have the semiorthogonal decomposition:
$$\Db(X) = \langle\ko(-h), \Phi' \Db(\PP^2,\kb_0), \ko, \ko(h) \rangle.$$
Perform the left mutation of $\ko(h)$ through its left orthogonal, which gives
$$\Db(X) = \langle\ko(-H), \ko(-h), \Phi' \Db(\PP^2,\kb_0), \ko\rangle,$$
using Lemma \ref{lem-mut} and $\omega_X = \ko(-H-h)$.
\begin{lemma}\label{complet-orth-1}
The pair $\langle\ko(-H),\ko(-h)\rangle$ is completely orthogonal.
\end{lemma}
\begin{proof}
Consider the semiorthogonal decomposition \ref{decomp-blow-quintic} and perform the left mutation of $\ko(H)$
through its left orthogonal. By Lemma \ref{lem-mut} we get
$$\Db(X) = \langle \ko(-h), \Psi \Db(\Gamma), \ko(-2H), \ko(-H), \ko \rangle,$$
which gives us $\Hom^{\bullet}(\ko(-H),\ko(-h))=0$ by semiorthogonality.
\end{proof}
We can now exchange $\ko(-H)$ and $\ko(-h)$, obtaining a semiorthogonal decomposition
$$\Db(X) = \langle\ko(-h), \ko(-H), \Phi' \Db(\PP^2,\kb_0), \ko\rangle.$$
The right mutation of $\ko(-h)$ through its right orthogonal gives (with Lemma \ref{lem-mut})
the semiorthogonal decomposition:
$$\Db(X) = \langle\ko(-H), \Phi' \Db(\PP^2,\kb_0), \ko, \ko(H)\rangle.$$
Perform the left mutation of $\Phi' \Db(\PP^2,\kb_0)$ through $\ko(-H)$. The functor $\Phi'' = \Phi' \circ L_{\ko(-H)}$
is full and faithful and we have the semiorthogonal decomposition:
$$\Db(X) = \langle\Phi'' \Db(\PP^2,\kb_0),\ko(-H), \ko, \ko(H)\rangle.$$
This shows, by comparison with (\ref{decomp-blow-quintic}),
that $\Phi''\Db(\PP^2,\kb_0) = \langle\Psi \Db(\Gamma), \ko(-2H)\rangle$. 
\end{proof}

\subsection{Degree four discriminant}

Suppose $C \subset \PP^2$ is a degree four curve with at most double points. We are going to describe $X$ as a hyperplane section
of a conic bundle over (a blow-up of) $\PP^3$, basing upon a construction from \cite{micheleconicbundle}.
Let $\Gamma$ be a smooth genus 2 curve, and $\Pic^n(\Gamma)$ the Picard variety of $\Gamma$ that parametrizes degree $n$
line bundles, up to linear equivalence. Since $g(\Gamma)=2$, $\Pic^1(\Gamma)$ contains the canonical Riemann theta divisor 
$\Theta:=\{L\in \Pic^1(\Gamma)|h^0(\Gamma,L)\neq 0\}$. It is well known that the Kummer surface $Kum(\Gamma):=\Pic^0(\Gamma)/\pm \id$
is naturally embedded in the linear system $|2\Theta|=\PP^3$ via the map 
\begin{eqnarray*}
Kum(\Gamma) & \rightarrow & |2\Theta|;\\
(\alpha \sim -\alpha) & \mapsto & \Theta_{\alpha} + \Theta_{-\alpha};
\end{eqnarray*}
where $\Theta_{\alpha}$ is the theta divisor translated by $\alpha$.
Hence the surface $Kum(\Gamma)$ sits in $\PP^3$ as a quartic surface 
with 16 double points corresponding to the 2-torsion. Note that the point corresponding to the line bundle $\ko_{\Gamma}$ is a node, and we will
call it \it the origin \rm or simply $\ko_{\Gamma}$.

Now we remark that $\Gamma$ is tri-canonically embedded in $\PP^4=|\omega_{\Gamma}^3|^*$, moreover we have a rational map
$$\varphi:\PP^4 \dashrightarrow \PP^3 :=|\cI_{\Gamma}(2)|^*$$
given by quadrics in the ideal of $\Gamma$.  In \cite{micheleconicbundle} it is shown that there exists an
isomorphism $|\cI_{\Gamma}(2)|^*\cong |2\Theta|$. Let now $\widetilde{Kum}(\Gamma)$ be the blow-up of $Kum(\Gamma)$
in the origin $\ko_{\Gamma}$ and $\widetilde{\PP^3}$ the corresponding blow up of $\PP^3$,
so that we have $\widetilde{Kum}(\Gamma)\subset \widetilde{\PP^3}$. Consider now the curve $\Gamma$ in $\PP^4$ and any point
$p\in \Gamma$. We denote by $q_p$ the only effective divisor in the linear system $|\omega_{\Gamma} (-p)|$
(notice that $q_p = \tau(p)$, where $\tau$ denotes the hyperelliptic involution). The ruled surface
$$S:=\{x\in \PP^4| \exists p\in \Gamma, x\in \overline{pq_p}\}$$
is a cone over a twisted cubic $Y$ in $\PP^3$ \cite[Prop1.2.2]{micheleconicbundle}. Let $\widetilde{\PP^4}$ the blow-up of $\PP^4$ along the cubic cone,
then the main result of \cite{micheleconicbundle} can be phrased as follows.

\begin{theorem}\label{bolos}

The rational map $\varphi$ resolves to a morphism $\widetilde{\varphi}:\widetilde{\PP^4} \rightarrow \widetilde{\PP^3}$ that is a
conic bundle degenerating on $\widetilde{Kum}(\Gamma)$. Hence we have the following commutative diagram.

$$\xymatrix{
\widetilde{\PP^4} \ar[r]^{\widetilde{\varphi}} \ar[d] & \widetilde{\PP^3} \ar[d] \\
\PP^4 \ar@{-->}[r]^{\varphi} & \PP^3 \\
}$$

\end{theorem}

\begin{remark}
The conic bundle described in Thm. \ref{bolos} is standard. This is straightforward from the description in \cite{micheleconicbundle}.
\end{remark} 

For any plane quartic curve $C$ with at most double points, we are going to obtain a structure of a standard
conic bundle on $\PP^2$ degenerating on $C$ by taking the restriction of $\widetilde{\varphi}$ to suitable hyperplanes $H\subset \PP^3$
for suitable choices of the genus two curve $\Gamma$ (and hence degenerating on the quartic curve $H\cap Kum(\Gamma)$). In fact every quartic curve with at most double points can be obtained via 
hyperplane intersection with an appropriate Jacobian Kummer surface, see \cite[Rem. 2.2]{almeida-gruson-perrin} 
and \cite{verraprym}. 

\begin{proposition}\label{prop-the-quartic-discriminant}
Let $X \to \PP^2$ be a standard conic bundle whose discriminant locus is a degree $4$ curve. Then there
exists a smooth genus $2$ curve $\Gamma$ and an embedding $\Gamma \subset Z$ into a smoth quadric threefold,
such that $X$ is isomorphic to the blow-up of $Z$ along $\Gamma$.
\end{proposition}

\begin{proof}
First, given $\Gamma$, we use Thm. \ref{bolos} to construct explicitly a conic bundle with the required properties.

Consider the composition $\phi:\widetilde{\PP^4} \rightarrow \widetilde{\PP^3} \rightarrow \PP^3$ of the conic 
bundle of Thm. \ref{bolos} with the blow-down map. Consider a hyperplane $N\subset \PP^3$ not containing 
the origin $\ko_{\Gamma}$ and denote by $X:= \phi^{-1}(N)$ and by $\pi$ the restriction of $\phi$ to $X$.
Then the induced map $\pi: X \rightarrow N \simeq \PP^2$ defines a standard conic bundle that degenerates on the intersection $N\cap Kum(\Gamma)$. Then it is easy to see that $X$ is isomorphic to the blow-up along $\Gamma$ of a smooth  quadric hypersurface $Z \subset \PP^4$ in the ideal of $\Gamma \subset\PP^4$. We remark that the quadric $Z$ is smooth because singular quadrics in the ideal of $\Gamma$ are cones over the quadrics in $\PP^3$ vanishing on the twisted cubic $Y$, and they correspond to hyperplanes in $\PP^3$ that contain $\ko_{\Gamma}$.
 
It is also known (\cite{almeida-gruson-perrin}, \cite{verraprym}) that the admissible double cover of 
$N\cap Kum(\Gamma)$ induced by the degree 2 cover $J(\Gamma)\rightarrow Kum(\Gamma)$ has Prym variety isomorphic to 
$J(\Gamma)$ and (\cite{verraprym}) that, by choosing appropriate genus 2 curves, one obtains all admissible double covers of all plane quartics this way.
Remark that this coincides in fact with the double cover of the plane quartic induced by the restriction of the conic bundle
degenerating on the Kummer variety. This means that the intermediate Jacobian $J(X)$ is isomorphic to $J(\Gamma)$ as a principally
polarized Abelian variety.

To complete the proof, we show that for any admissible double cover of a plane quartic, the corresponding standard
conic bundle can be obtained by the previous construcion. Remark indeed that we can always assume that the quadric 
hypersurface $Z$ we are considering is smooth:
by using the invariance of $Kum(\Gamma)$ under the action of $(\ZZ/2\ZZ)^4$ we can just not consider hyperplanes passing through the origin (whose inverse images are singular quadrics). Indeed, up to the choice of the appropriate curve $\Gamma$, one can obtain any plane quartic with at most double points by considering the intersection of hyperplanes in $\PP^3$, that do not contain the origin, with the Kummer quartic.
\end{proof}

Resuming, let $\chi: X \to Z$ be the blow-up of $Z$ along $\Gamma$. Let us denote by $\ko(H)$ both the restriction of
$\ko_{\PP^4}(1)$ to $Z$ and its pull-back to $X$ via $\chi$, and by $D$ the exceptional divisor.
Remark that $\omega_Z = \ko(-3H)$. Let $\Sigma$ be the spinor bundle on the quadric $Z$.
Let $\pi: X \to \PP^2$ be the conic bundle
structure. We denote by $\ko(h)$ the pull back of $\ko_{\PP^2}(1)$ via $\pi$. The construction of the
map $\pi$ gives $\ko(h)=\ko(2H-D)$, then we have $\ko(D) = \ko(2H-h)$. The canonical bundle $\omega_X$ is given
by $\ko(-3H+D)= \ko(-H-h)$.

\begin{proposition}\label{prop:1stcase}
Let $\pi:X \to \PP^2$ be a standard conic bundle whose discriminant locus $C$ is a degree 4 curve.
Then there exists an exceptional
object $E$ in $\Db(\PP^2,\kb_0)$ such that (up to equivalences):
$$\Db(\PP^2,\kb_0) = \langle \Db(\Gamma), E\rangle,$$
where $\Gamma$ is a smooth projective curve such that $J(X) \simeq J(\Gamma)$ as a principally polarized Abelian variety.
\end{proposition}

\begin{proof}
Consider the blow-up $\chi: X \to Z$. Orlov formula (see Prop. \ref{decomp-blow-up}) provides
a fully faithful functor $\Psi: \Db(\Gamma) \to \Db(X)$ and a semiorthogonal decomposition:
$$
\Db(X) = \langle \Psi \Db(\Gamma), \chi^* \Db(Z) \rangle.
$$
By \cite{kapranovquadric}, the derived category $\Db(Z)$ has a full exceptional sequence
$\langle \Sigma (- 2H), \ko(-H), \ko, \ko(H) \rangle$. We get then the semiorthogonal
decomposition:
\begin{equation}\label{decomp-blow-quartic}
\Db(X) = \langle \Psi \Db(\Gamma), \Sigma (-2H), \ko(-H), \ko, \ko(H) \rangle.
\end{equation}
Kuznetsov formula (see Prop. \ref{decom-conic-bd}) provides the decomposition:
$$
\Db(X) = \langle \Phi \Db(\PP^2,\kb_0), \pi^* \Db(\PP^2) \rangle.
$$
The derived category $\Db(\PP^2)$ has a full exceptional sequence
$\langle\ko_{\PP^2}(-1),\ko_{\PP^2}, \ko_{\PP^2}(1) \rangle$. We get then the semiorthogonal decomposition:
\begin{equation}\label{decomp-conic-quartic}
\Db(X) = \langle \Phi \Db(\PP^2,\kb_0), \ko(-h), \ko, \ko(h) \rangle.
\end{equation}
We perform now some mutation to compare the decompositions (\ref{decomp-blow-quartic}) and (\ref{decomp-conic-quartic}).
Surprisingly to us, we will follow the same path as in the proof of Proposition \ref{prop:2ndcase}.

Consider the decomposition (\ref{decomp-conic-quartic}) and mutate $\Phi \Db(\PP^2,\kb_0)$ to the right through $\ko(-h)$.
The functor $\Phi' = \Phi \circ R_{\ko(-h)}$ is full and faithful and we have the semiorthogonal decomposition:
$$\Db(X) = \langle\ko(-h), \Phi' \Db(\PP^2,\kb_0), \ko, \ko(h) \rangle.$$
Perform the left mutation of $\ko(h)$ through its left orthogonal, which gives
$$\Db(X) = \langle\ko(-H), \ko(-h), \Phi' \Db(\PP^2,\kb_0), \ko\rangle,$$
using Lemma \ref{lem-mut} and $\omega_X = \ko(-H-h)$.
We can prove the following Lemma in the same way we proved Lemma \ref{complet-orth-1}.
\begin{lemma}\label{complet-orth-2}
The pair $\langle\ko(-H),\ko(-h)\rangle$ is completely orthogonal.
\end{lemma}
We can now exchange $\ko(-H)$ and $\ko(-h)$, obtaining a semiorthogonal decomposition
$$\Db(X) = \langle\ko(-h), \ko(-H), \Phi' \Db(\PP^2,\kb_0), \ko\rangle.$$
The right mutation of $\ko(-h)$ through its right orthogonal gives (with Lemma \ref{lem-mut})
the semiorthogonal decomposition:
$$\Db(X) = \langle\ko(-H), \Phi' \Db(\PP^2,\kb_0), \ko, \ko(H)\rangle.$$
Perform the left mutation of $\Phi' \Db(\PP^2,\kb_0)$ through $\ko(-H)$. The functor $\Phi'' = \Phi' \circ L_{\ko(-H)}$
is full and faithful and we have the semiorthogonal decomposition:
$$\Db(X) = \langle\Phi'' \Db(\PP^2,\kb_0),\ko(-H), \ko, \ko(H)\rangle.$$
This shows, by comparison with (\ref{decomp-blow-quartic})
that $\Phi''\Db(\PP^2,\kb_0) = \langle\Psi \Db(\Gamma), \Sigma(-2H)\rangle$.
\end{proof}

\subsection{Degree three discriminant}

Let $\pi: X \to \PP^2$ be a standard conic bundle whose discriminant $C$ is a cubic curve.
Consider
$X \subset \PP^2 \times \PP^2$ a hypersurface of bidegree $(1,2)$. The map $\pi$  given by the restriction of the first
projection $p_1: \PP^2 \times \PP^2 \to \PP^2$ is a conic bundle degenerating on a cubic curve. 
More precisely, any cubic curve $C$ with nodes as singularities can be written as a determinant of a $3\times 3$ symmetric matrix of linear forms, with rank dropping on the nodes. In this way one obtains the bidegree $(1,2)$ hypersurface which is the standard conic bundle degenerating on $C$. The restriction of the second projection gives a $\PP^1$-bundle structure $p: X \to \PP^2$. Remark that the intermediate Jacobian $J(X)$ is trivial.

Let $\ko(h):= \pi^* \ko_{\PP^2}(1)$ 
and $\ko(H) := p^* \ko_{\PP^2}(1)$, then $\ko(H) = \ko_{\pi}(1)$ and $\ko(h) = \ko_{p}(1)$.
We have the canonical bundle $\omega_X = \ko(-2h - H)$.

\begin{proposition}
Let $\pi:X \to \PP^2$ be a standard conic bundle whose discriminant locus $C$ is a degree 3 curve. Then there exist three exceptional
objects $E_1$, $E_2$ and $E_3$ in $\Db(\PP^2,\kb_0)$ such that (up to equivalences):
$$\Db(\PP^2,\kb_0) = \langle E_1, E_2, E_3 \rangle.$$
\end{proposition}

\begin{proof}
Consider the $\PP^1$-bundle structure $p: X \to \PP^2$. Then by Proposition \ref{prop-decom-projbund} we have
$$\Db(X) = \langle p^* \Db(\PP^2), p^* \Db(\PP^2) \otimes \ko_p(1)\rangle,$$
which gives, recalling that $\ko(h)=\ko_p(1)$,
\begin{equation}\label{deco-cubic-1}
\Db(X) = \langle \ko(-2H), \ko(-H), \ko, \ko(h-H), \ko(h), \ko(h+H) \rangle,
\end{equation}
where we used the decompositions $\langle \ko_{\PP^2}(-2), \ko_{\PP^2}(-1), \ko_{\PP^2} \rangle$
and $\langle \ko_{\PP^2}(-1), \ko_{\PP^2}, \ko_{\PP^2}(1) \rangle$ in the first and in the second
occurrence of $p^*\Db(\PP^2)$ respectively.

Kuznetsov formula (see Prop. \ref{decom-conic-bd}) provides the decomposition
$$
\Db(X) = \langle \Phi \Db(\PP^2,\kb_0), \pi^*\Db(\PP^2) \rangle,
$$
which, choosing the decomposition $\Db(\PP^2) = \langle \ko_{\PP^2}(-1), \ko_{\PP^2}, \ko_{\PP^2}(1) \rangle$, gives
\begin{equation}\label{deco-cubic-2}
\Db(X) = \langle \Phi \Db(\PP^2,\kb_0), \ko(-h),\ko,\ko(h)\rangle.
\end{equation}
We perform now some mutation to compare the decomposition (\ref{deco-cubic-1}) and (\ref{deco-cubic-2}).

Consider the decomposition \ref{deco-cubic-1} and mutate $\ko(h-H)$ to the left through $\ko$. This gives
$$\Db(X) = \langle \ko(-2H), \ko(-H), E, \ko, \ko(h), \ko(h+H) \rangle,$$
where $E:=L_{\ko}\ko(h-H)$ is an exceptional object. Perform the left mutation of $\ko(H+h)$ through its left orthogonal,
which gives
$$\Db(X)  = \langle \ko(-h), \ko(-2H), \ko(-H), E, \ko, \ko(h) \rangle,$$
using Lemma \ref{lem-mut} and $\omega_X = \ko(-2h-H)$. Finally, mutate
the exceptional sequence $(\ko(-2H),\ko(-H),E)$  to the left through $\ko(-h)$. This gives
$$\Db(X) = \langle E_1, E_2, E_3, \ko(-h),\ko,\ko(h) \rangle,$$
where $(E_1, E_2, E_3) := L_{\ko(-h)}((\ko(-2H),\ko(-H),E))$ is an exceptional sequence. This shows, by comparison with
(\ref{deco-cubic-2}), that $\Phi \Db(\PP^2,\kb_0)= \langle E_1,E_2,E_3 \rangle$.
\end{proof}

\begin{remark}
Note that, as pointed out to us by A.Kuznetsov, the same result can be obtained using \cite[Thm. 5.5]{kuznetconicbundles}: since the
total space $X$ is smooth, the complete intersection of the net of conics is empty.
Recall from \cite{kuznetconicbundles} that we can define the sheaves $\kb_{2i}= \kb_0 \otimes \ko(i)$ and $\kb_{2i+1}=
\kb_1 \otimes \ko(i)$, and $\{\kb_i\}_{i=-3}^{-1}$ give a full exceptional collection for $\Db(\PP^2,\kb_0)$.
This follows indeed from \cite[Thm 5.5]{kuznetconicbundles}.
\end{remark}

\section{Rational conic bundles over Hirzebruch surfaces}\label{sect-hirze}

Let us consider now the case $S=\FF_n$. In this case, following (\cite{iskovconicduke, shokuprym}),
we have only two non-trivial possibilities for a standard conic bundle $\pi:X \to S$ to be rational: there must exist
a base point free pencil $L_0$ of rational curves such that either $L_0\cdot C = 3$ or $L_0\cdot C = 2$. In the first
case $C$ is trigonal, and in the second one $C$ is hyperelliptic. In both instances, the only such pencil is the natural
ruling of $S$. Hence, if we let $q: S \to \PP^1$ be the ruling map, the trigonal or hyperelliptic structure is
induced by the fibers of $q$.
As proved in Lemma \ref{lem-square-equiv}, once we fix the discriminant curve and the associated double
cover, we fix the Clifford algebra $\kb_0$. We then construct for
any such curve and associated double cover a model of rational standard conic bundle $X$ for which we provide the required
semiorthogonal decomposition.

We will consider relative realizations of the following classical construction: let $x_1, \ldots, x_4$ be four
points in general position in a projective plane $\PP^2$, and consider the pencil of all the plane conics passing through
$x_1, \ldots, x_4$. Blowing up $\PP^2$ along the four points, we get a conic bundle $Y \to \PP^1$ with degree three discriminant.
It is not difficult to see that a relative version of the preceding construction involves a $\PP^2$-bundle over $\PP^1$ containing a tetragonal curve $\Gamma \to \PP^1$, and the absolute construction performed fiberwise will give, after blow-up of $\Gamma$, a conic bundle $Y \to S$, for $S$ a Hirzebruch surface. In fact, we will show that, if $\Gamma$ is connected, then $Y$ is standard and has trigonal discriminant, which gives the first case.
If $\Gamma$ is the disjoint union of two hyperelliptic curves, the conic bundle $Y \to S$ is not standard and has discriminant the disjoint union of a hyperelliptic curve and a line. We will then perform a birational
transformation to obtained a standard conic bundle $\pi: X \to S$ with hyperelliptic discriminant. More precisely: fixed the discriminant curve and the double cover $\tilde{C} \to C$, we describe
a structure of conic bundle $\pi: X \to S$ following Casnati \cite{casnati} as the blow-up of a $\PP^2$-bundle
over $\PP^1$ along a certain tetragonal curve given by Recillas' construction
(\cite{recillas} for the trigonal case) or one of its degenerations (for the hyperelliptic case). These constructions can be performed for all the trigonal or hyperelliptic discriminant curves with at most nodes as singularities.
We describe the case of trigonal and hyperelliptic discriminant separately, following anyway the same path.  

The trigonal construction had already been used in the framework of conic bundles, in a slightly different context, in \cite{bertbrau}.

\subsection{Trigonal discriminant}\label{subsect-trigonal}

In the case where $C$ is a trigonal curve on $S$, we can give an explicit description of the conic bundle $\pi: X \to S$
degenerating along $C$, exploiting Recillas' trigonal construction \cite{recillas}.
We will develop the trigonal construction in the more general framework presented by Casnati in \cite{casnati}, that emphasizes
the conic bundle structure. For a detailed account in the curve case, with emphasis on the beautiful consequences on the structure
of the Prym map, see also \cite{donagi-prym}.

Before going through details let us recall from \cite{casnati-ekedhal} that any Gorenstein degree 3 cover $t':C \to \PP^1$ 
can be obtained inside a suitable $\PP^1$-bundle $S:= \PP(\kf)$ over $\PP^1$ as the zeros of a relative cubic 
form in two variables. On the other hand each Gorenstein degree 4 cover $t: \Gamma \to \PP^1$ is obtained 
\cite{casnati-ekedhal} as the base locus of a relative pencil of conics over $\PP^1$ contained in a $\PP^2$-bundle 
$Z:=\PP(\kg)$ over $\PP^1$. Moreover, the restriction, both to $C$ and $\Gamma$, of the natural projection of each projective 
bundle gives the respective finite cover map to $\PP^1$. For instance, the $\PP^2$-fiber $Z_x$ contains the four points of $\Gamma$ 
over $x\in \PP^1$. The first result of this section proves that any standard conic bundle with trigonal discriminant over a Hirzebruch
surface is realized naturally via this construction.

\begin{proposition}\label{prop:cb-with-trigo}
Let $\pi: X \to S$ be a standard conic bundle with trigonal discriminant curve $C$, where the trigonal structure $C \to \PP^1$ is induced by
the $\PP^1$-bundle structure $S \to \PP^1$. Then there exists a $\PP^2$-bundle $Z \to \PP^1$ containing
a tetragonal curve $\Gamma \to \PP^1$ whose tetragonal structure is induced by $Z \to \PP^1$ and such that $X$ is isomorphic
to the blow-up of $Z$ along $\Gamma$.
\end{proposition}

\begin{proof} 
First of all, given any admissible cover of a trigonal curve we describe a conic bundle structure over a Hirzebruch surface
as the blow-up of the $\PP^2$-bundle containing the tetragonal curve.

Consider the trigonal curve $C \to \PP^1$ sitting in the $\PP^1$-fibration $\PP(\kf) \to \PP^1$ and
the tetragonal curve $\Gamma$ inside the $\PP^2$-fibration $Z \to \PP^1$, associated to $C$ by the Casnati-Recillas construction.
Fix a point $x$ of $\PP^1$ and the corresponding fiber $Z_x$.
In this plane, consider the pencil of conics through the 4 points
$Z_x \cap \Gamma$. For each point of $\PP^1$, we obtain a pencil of conics with three degenerate conics. 
We then have a pencil of such conic pencils (parameterized by the ruled surface $S=\PP(\kf)$), which can be described as the 
2-dimensional family of vertical conics in $Z$ intersecting $\Gamma$ in all the four points of $\Gamma \cap Z_x$.
The standard conic bundle over $S$ is then given by resolving the linear system $\vert
\ko_{Z/\PP^1}(2) - \Gamma \vert$. That is, the blow-up of $Z$ along $\Gamma$ is a standard conic
bundle $\pi:X \to S$ with the required degeneration.

Thanks to Thm. 6.5 of \cite{casnati} (see also Thm 2.9 of \cite{donagi-prym}), to any 
trigonal Gorenstein curve $C$ with an admissible double cover,
we can associate a smooth tetragonal curve $\Gamma$ such that $C$ is the discriminant 
locus of the conic bundle that defines $\Gamma$. More precisely,
Let $X$ be the relative pencil of conics in the projective
bundle $Z\to \PP^1$ defined by $\Gamma$. We have a $\PP^1$-bundle $S$ over $\PP^1$,
a conic bundle structure $\pi:X \to S$ with discriminant curve $C$ in its natural embedding as a relative cubic form.

The proof is completed by showing that any conic bundle with trigonal discriminant can be realized in this way, but
Theorem 6.5 \cite{casnati} ensures that all 
trigonal curves with at most double points that sit in some ruled surface $S$ are discriminant divisors of a 
conic bundle (for details see \cite{casnati}, Sect. 5 and 6).
\end{proof}

Notice that this tight connection between trigonal and 
tetragonal curves is reflected also when considering the corresponding Prym and Jacobian varieties. The Prym variety 
of the admissible cover of $C$ induced by the conic bundle is in fact isomorphic to the Jacobian of $\Gamma$
\cite{recillas} and to the 
intermediate Jacobian of the conic bundle $X$.

Summarizing, we end up with the following commutative diagram:

$$\xymatrix{
& X \ar[ld]_{\pi} \ar[rd]^{\chi} & D \ar@{_{(}->}[l] \ar[rd]^{\chi}\\
S \ar[rd]_q & & Z \ar[ld]^p & \Gamma \ar@{_{(}->}[l] \\
& \PP^1,
}\label{trigodiagram}$$

where $p:Z \to \PP^1$ is a $\PP^2$-bundle, $\Gamma \subset Z$ the tetragonal curve,
$\chi: X \to Z$ the
blow-up of $\Gamma$ with exceptional divisor $D$. The surface $q: S \to \PP^1$ is ruled and $\pi: X \to S$ is the
conic bundle structure degenerating along the trigonal curve $C$.

We denote by $\ko(H) := \ko_{Z/\PP^1}(1)$ the relative ample line bundle on $Z$ and by $\ko(h) := \ko_{S/\PP^1}(1)$
the relative ample
line bundle on $S$. By abuse of notation, we still denote by $\ko(H)$ and $\ko(h)$ the pull-back of $\ko(H)$ and $\ko(h)$
via $\chi$
and $\pi$ respectively. The construction of the map $\pi$ gives $\ko(h) = \ko(2H - D)$,
from which we deduce that $\ko(D) = \ko(2H -h)$.
The canonical bundle $\omega_X$ is given by $\omega_X = \chi^* \omega_Z + \ko(D)$. Since we
have $\omega_Z = \omega_{Z/\PP^1} + p^* \omega_{\PP^1} = \ko(-3H)+ p^* 
\det(\ke^*) + p^* \omega_{\PP^1}$, we finally get
$\omega_X = \ko(-H -h) + \chi^* p^* \omega_{\PP^1}(\det(\ke^*))$.

\begin{proposition}\label{prop:trig-case}
Let $\pi:X \to S$ be a conic bundle whose discriminant locus $C$ is a trigonal curve whose trigonal structure is
given by the intersection of $C$ with the ruling $S \to \PP^1$.
Then there exist two exceptional
objects $E_1, E_2$ in $\Db(S,\kb_0)$ such that (up to equivalences):
$$\Db(S,\kb_0) = \langle \Db(\Gamma), E_1, E_2\rangle,$$
where $\Gamma$ is a smooth projective curve such that $J(X) \simeq J(\Gamma)$ as a principally polarized Abelian variety.
\end{proposition}

\begin{proof}
Consider the blow-up $\chi: X \to Z$. Orlov formula (see Prop. \ref{decomp-blow-up}) provides
a fully faithful functor $\Psi: \Db(\Gamma) \to \Db(X)$ and a semiorthogonal decomposition:
$$\Db(X) = \langle \Psi \Db(\Gamma), \chi^* \Db(Z) \rangle.$$
By Prop. \ref{prop-decom-projbund} we can choose the semiorthogonal decomposition
$\langle p^* \Db(\PP^1) \otimes \ko(- H), p^* \Db(\PP^1), p^* \Db(\PP^1) \otimes \ko(H) \rangle$
of $\Db(Z)$. We then get:
\begin{equation}\label{decomp-blow-trigonal}
\Db(X) = \langle \Psi \Db(\Gamma),\ \chi^* p^* \Db(\PP^1) \otimes \ko(- H),\
\chi^* p^* \Db(\PP^1),\ \chi^* p^* \Db(\PP^1) \otimes \ko(H) \rangle.
\end{equation}
Kuznetsov formula (see Prop. \ref{decom-conic-bd}) provides the decomposition:
$$\Db(X) = \langle \Phi \Db(S,\kb_0), \pi^* \Db(S) \rangle.$$
By Prop. \ref{prop-decom-projbund} we can choose the semiorthogonal decomposition
$\langle q^* \Db(\PP^1) \otimes \ko(-h), q^* \Db(\PP^1) \rangle$ of $\Db(S)$. We then
get:
\begin{equation}\label{decomp-conic-trigonal}
\Db(X) = \langle \Phi \Db(S,\kb_0),\ \pi^* q^* \Db(\PP^1) \otimes \ko(-h),\ \pi^* q^* \Db(\PP^1) \rangle.
\end{equation}
We perform now some mutation to compare the decompositions (\ref{decomp-blow-trigonal}) and (\ref{decomp-conic-trigonal}).
First of all, since $\pi^* q^* = \chi^* p^*$, we have $\pi^* q^* \Db(\PP^1) = \chi^* p^* \Db(\PP^1)$ and we will
denote this category simply by $\Db(\PP^1)$.

Consider the decomposition (\ref{decomp-conic-trigonal}) and mutate $\Phi \Db(S,\kb_0)$ to the right with respect
to $\Db(\PP^1) \otimes \ko(-h)$. The functor $\Phi' := \Phi \circ R_{\Db(\PP^1)\otimes \ko(-h)}$ is full and faithful
and we have the semiorthogonal decomposition
$$\Db(X) = \langle \Db(\PP^1)\otimes \ko(-h),\ \Phi' \Db(S,\kb_0),\ \Db(\PP^1) \rangle.$$
Perform the right mutation of $\Db(\PP^1)\otimes \ko(-h)$ through its right orthogonal. We have 
$$R_{<\Db(\PP^1)\otimes \ko(-h)>^{\perp}} (\Db(\PP^1)\otimes \ko(-h))
= \Db(\PP^1)\otimes \ko(-h) \otimes \omega_X^* \cong \Db(\PP^1) \otimes \ko(H).$$
Indeed $\omega_X = \ko(-H - h) + \chi^* p^* \omega_{\PP^1}(\det(\ke^*))$ and the
tensorization with $\chi^* p^* \omega_{\PP^1}(\det(\ke^*))$ gives an autoequivalence of $\Db(\PP^1)$. We then have the
decomposition
$$\Db(X) = \langle \Phi' \Db(S,\kb_0),\ \Db(\PP^1),\ \Db(\PP^1) \otimes\ko(H) \rangle.$$
Comparing this last decomposition with (\ref{decomp-blow-trigonal}) we get
$$\Phi' \Db(S,\kb_0) = \langle \Psi(\Gamma), \Db(\PP^1) \otimes \ko(-H) \rangle,$$
and the proof now follows recalling that $\Db(\PP^1)$ has a two-objects full exceptional sequence.
\end{proof}

\subsection{Hyperelliptic discriminant}

Also in the case where $C$ is a hyperelliptic curve on $S$, we can give an explicit description of the conic bundle $\pi: X \to S$
degenerating along $C$, exploiting Mumford's construction \cite{mumfordprym}: indeed giving a double cover $\tilde{C} \to C$
is equivalent to splitting the set $R$ of ramification points of $C$ into two sets $R_0$ and $R_1$ with an even number of
elements. Taking the hyperelliptic curves $\Gamma_0$ and $\Gamma_1$ ramified respectively on $R_0$ and $R_1$,
we have $\tilde{C} = \Gamma_0 \times_{\PP^1} \Gamma_1$ and $P(\tilde{C}/C) \simeq J(\Gamma_0) \oplus J(\Gamma_1)$.
The key remark here is that $X$ can be obtained via a birational transformation starting from a degenerate case of the 
Casnati-Recillas construction. Indeed, consider $\Gamma := \Gamma_0 \amalg \Gamma_1$. We have $J(\Gamma) = J(\Gamma_0) \oplus J(\Gamma_1)$
and a natural tetragonal structure on $\Gamma \to \PP^1$. As in the previous section, the results from \cite{casnati-ekedhal}
provide us a $\PP^2$-bundle $Z \to \PP^1$ and a Hirzebruch surface $S$ containing a trigonal curve $C'$.

\begin{lemma}\label{lem-the-nonstandard}
The curve $C'$ is the disjoint union of a line $L$ and a hyperelliptic curve $C \to \PP^1$ with ramification $R= R_0 \cup R_1$.
The double cover $\tilde{C'}\to C'$ is trivial along $L$.
The blow-up of $Z$ along $\Gamma$ gives a non-standard conic bundle structure $\pi':Y \to S$ whose discriminant
double cover is $\tilde{C'}\to C'$.
\end{lemma}
\begin{proof}
To construct of the conic bundle $\pi': Y \to S$ we proceed in the same way as in Proposition \ref{prop:cb-with-trigo}.
As Donagi pointed out \cite[Ex. 2.10]{donagi-prym}, in the Casnati-Recillas construction the tetragonal curve
$\Gamma=\Gamma_0\amalg \Gamma_1$ corresponds to the trigonal curve $C'$. 

On the other hand the double cover $\tilde{C}'$ of $C'$ splits as $\tilde{C}\amalg \PP^1 \amalg \PP^1$, where $\tilde{C}$
is a double cover of $C$. The $\PP^1 \amalg \PP^1$ part is the trivial disconnected double cover of $L$. 
This implies that the conic bundle $\pi':Y \to S$ obtained as the blow up of $Z$ along
$\Gamma$ is not standard.
Indeed, being the double cover of $L$ trivial implies that the preimage $G$ of $L$ inside $Y$
is the union of two Hirzebruch surfaces intersecting along a line.
\end{proof}

Notice that, since the double cover of $C'$ is trivial along $L$,
we have an isomorphism of principally polarized abelian varieties $P(\tilde{C}'/C')\cong P(\tilde{C}/C)$.
The standard conic bundle $\pi:X \to S$ degenerating along the hyperelliptic curve $C$ should then be described
as a birational transformation of $\pi':Y \to S$, smoothing the preimage $G$ of $L$.
Such birational transformation is a slight generalization of the elementary transformation
described in \cite[Sect. 2.1]{sarkisov-bira} and, roughly, it consists in contracting one of 
the two components of $G$ onto the intersection of the two.
After that, $L$ is no longer contained in the discriminant locus, hence the discriminant locus
is the hyperelliptic curve. This transformation corresponds to a birational transformation of the projective bundle $Z$.

In order to do that, we need to make a choice between the two disjoint components of $\Gamma= \Gamma_0 \amalg \Gamma_1$.
At the end of the section, we will see that this choice does not affect the required description of the derived category,
up to autoequivalences (see Remark \ref{rem-the-same}).
Let $\widetilde{Z} \to Z$ be the blow-up of $Z$ along $\Gamma_0$ and $\widetilde{D}$ the exceptional divisor.
Let $T\subset \widetilde{Z}$ be the strict transform of the ruled
surface obtained by taking the closure of the locus of lines spanned by each couple of points of $\Gamma_0$ associated
by the hyperelliptic involution.  
Let us denote $Q$ the 3-fold obtained from $\widetilde{Z}$ by blowing down $T$ to a line along the ruling.
Remark that, since $\Gamma_1$ is disjoint from $\Gamma_0$, then $\Gamma_1$ is embedded in $Q$.

\begin{lemma}\label{lem-Q-is-qb}
{\rm (i)}  There is a quadric bundle structure $\tau: Q \to \PP^1$ of relative dimension 2, with simple degeneration along the
ramification set $R_0$ of $\Gamma_0$. There is a natural embedding $\Gamma_1 \subset Q$.

{\rm(ii)} There exists a full and faithful functor $\bar{\Psi}_0 : \Db(\Gamma_0) \to \Db(Q)$ and a semiorthogonal decomposition
$$\Db(Q) = \langle \bar{\Psi}_0 \Db(\Gamma_0),\ \tau^* \Db(\PP^1),\ \tau^* \Db(\PP^1)\otimes \ko_{Q/\PP^1}(1)>.$$
\end{lemma}

\begin{proof}
(i) Consider a point $x$ in $\PP^1$ and the fiber $Z_x$, which is a projective plane. Let $a_i$, $b_i$ be
the points where $\Gamma_i$ intersects $Z_x$. Then if we blow-up $a_0$ and $b_0$ and we contract the line
through them, we get a birational map $Z_x \dashrightarrow Q_x$, where $Q_x$ is a quadric surface,
which is smooth if and only if $a_0 \neq b_0$ \cite[pag. 85]{harris-first-course}
and has simple degeneration otherwise, in fact $\mathbb{F}_2$ is isomorphic to the blow up of a quadric cone in its node.

(ii) By \cite{kuznetconicbundles}, if we denote by $\kc_0$ the sheaf
of even parts of the Clifford algebra associated to $\tau$, there is a fully faithful functor $\bar{\Psi}_0:
\Db(\PP^1,\kc_0) \to \Db(Q)$ and a semiorthogonal decomposition
$$\Db(Q) = \langle \bar{\Psi}_0 \Db(\PP^1,\kc_0),\ \tau^* \Db(\PP^1),\ \tau^* \Db(\PP^1) \otimes \ko_{Q/\PP^1}(1)>.$$
Now apply \cite[Cor. 3.14]{kuznetconicbundles} to get the equivalence $\Db(\PP^1,\kc_0) \cong \Db(\Gamma_0)$.
\end{proof}

\begin{proposition}\label{prop-hyperell-conic-bd}
Let $\pi:X \to S$ be a standard conic bundle over a Hirzebruch surface with hyperelliptic discriminant. Let $R \subset
\PP^1$ be the ramification locus of $C$, $\Gamma_0$ and $\Gamma_1$ two hyperelliptic curves, with ramification loci
$R_0$ and $R_1$ respectively, such that $R=R_0 \cup R_1$ and $J(X) \simeq P(\tilde{C}/C) = J(\Gamma_0) \oplus J(\Gamma_1)$.
Then there is a quadric surface bundle $Q \to \PP^1$ degenerating along $R_0$, such that $\Gamma_1 \subset Q$ and
$X$ is isomorphic to the blow-up of $Q$ along $\Gamma_1$.
\end{proposition}
\begin{proof}
Given $\Gamma = \Gamma_0 \amalg \Gamma_1$,
consider the trigonal curve $C' = C \amalg \PP^1$ and the conic bundle $\pi':Y \to S$ from Lemma \ref{lem-the-nonstandard},
the blow-up of $Z$ along $\Gamma$. 
As proved in Lemma \ref{lem-Q-is-qb} the $\PP^2$-bundle $Z$ has been transformed into the quadric bundle $Q$ containing
the curve $\Gamma_1$. For any point $x$ of $\PP^1$, the pencil of conics in $Z_x$  passing through
$a_0,b_0,a_1,b_1$ has been transformed into a pencil of hyperplane sections of $Q_x$ passing through $a_1$ and $b_1$.
Hence each line of the ruling of $S$ corresponds to a pencil of quadratic hyperplane sections.
Moreover the conics over the rational curve $L\subset S$ had simple degeneration in $Y$ and are smooth in $X$.
On the rest of the ruled surface $S$ the degeneration type of the conics is preserved.

This implies that $\pi:X\to S$ is a standard conic bundle degenerating along the hyperelliptic $C$.
It is given by resolving the relative linear system $\vert \ko_{Q/\PP^1}(1)- \Gamma_1 \vert$.

Moreover, to any conic bundle with hyperelliptic discriminant and double cover $\tilde{C} \to C$, one
can associate the tetragonal curve $\Gamma$ and perform this construction.
\end{proof}
Summarizing, we end up with the following diagram:

$$\xymatrix{
& X \ar[dl]_{\pi} \ar[rd]^{\chi} & D \ar@{_{(}->}[l] \ar[rd]^{\chi}\\
S \ar[rd]_q & & Q \ar[ld]^{\tau} & \Gamma_1 \ar@{_{(}->}[l]\\
& \PP^1,
}$$

Where $\tau: Q \to \PP^1$ is a quadric bundle degenerating
exactly in the ramification locus of $\Gamma_0 \to \PP^1$ and contains the hyperelliptic curve $\Gamma_1$.
The map $\chi$ is the blow-up of $Q$ along $\Gamma_1$ with exceptional divisor $D$. The surface $q: S \to \PP^1$ is
ruled and $\pi: X \to S$ is the
conic bundle structure degenerating along the hyperelliptic curve $C$. Remark that $J(X)$ is isomorphic to $J(\Gamma_0) \oplus
J(\Gamma_1)$ as principally polarized Abelian variety. Since $J(X) \cong P(\tilde{C}/C)$, if $C$ is smooth it can be shown that $R_1 \cup R_0=R$ and the configuration of Pryms and Jacobians is the one described by Mumford in \cite{mumfordprym}.

We denote by $\ko(H):=\ko_{Q/\PP^1}(1)$ the relative ample line bundle on $Q$. We have $\omega_{Q/\PP^1} = \ko(-2H)$.
Denote by $\ko(h) := \ko_{S/\PP^1}(1)$ the relative ample
line bundle on $S$. By abuse of notation, we still denote by $\ko(H)$ and $\ko(h)$ the pull-backs of $\ko(H)$
and $\ko(h)$ via $\chi$
and $\pi$ respectively. The construction of the map $\pi$ gives $\ko(h) = \ko(H - D)$, from which we deduce
that $\ko(D) =\ko(H -h)$.

The canonical bundle $\omega_X$ is given by $\omega_X = \chi^* \omega_Q + D$. Since we
have $\omega_Q = \omega_{Q/\PP^1} + \tau^* \omega_{\PP^1}$, we finally get
$\omega_X = \ko(-H -h) + \chi^* \tau^* \omega_{\PP^1}$.

\begin{proposition}\label{prop:hyperell-case}
Let $\pi:X \to S$ be a conic bundle whose discriminant locus $C$ is a hyperelliptic curve whose hyperelliptic structure is
given by the intersection of $C$ with the ruling $S \to \PP^1$.
Then (up to equivalences):
$$\Db(S,\kb_0) = \langle \Db(\Gamma_1), \Db(\Gamma_0) \rangle,$$
where $\Gamma_0$ and $\Gamma_1$ are smooth projective curves such that $J(X) \simeq J(\Gamma_0) \oplus J(\Gamma_1)$
as a principally polarized Abelian variety.
\end{proposition}

\begin{proof}
Consider the blow-up $\chi: X \to Q$. Orlov formula (see Prop. \ref{decomp-blow-up}) provides a fully faithful
functor $\Psi_1: \Db(\Gamma_1) \to \Db(X)$ and a semiorthogonal decomposition:
$$\Db(X) = \langle \Psi_1 \Db(\Gamma_1), \chi^* \Db(Q) \rangle.$$
Lemma \ref{lem-Q-is-qb} gives us
\begin{equation}\label{decomp-blow-hyperell}
\Db(X) = \langle \Psi_1 \Db(\Gamma_1),\ \Psi_0 \Db(\Gamma_0),\ \chi^* \tau^* \Db(\PP^1),\
\chi^* \tau^* \Db(\PP^1)\otimes \ko(H) \rangle,
\end{equation}
where $\Psi_0 = \bar{\Psi}_0 \circ \chi^*$ is fully faithful.

Kuznetsov formula (see Prop. \ref{decom-conic-bd}) provides the decomposition:
$$\Db(X) = \langle \Phi \Db(S,\kb_0), \pi^* \Db(S) \rangle.$$
By Prop. \ref{prop-decom-projbund} we can choose the semiorthogonal decomposition
$\langle q^* \Db(\PP^1), q^* \Db(\PP^1)\otimes\ko(-h) \rangle$ of $\Db(S)$. We then
get
\begin{equation}\label{decomp-conic-hyperell}
\Db(X) = \langle \Phi \Db(S,\kb_0),\ \pi^* q^* \Db(\PP^1)\otimes\ko(-h),\ \pi^* q^* \Db(\PP^1) \rangle.
\end{equation}
We perform now some mutation to compare the decompositions (\ref{decomp-blow-hyperell}) and (\ref{decomp-conic-hyperell}).
First of all, since $\pi^* q^* = \chi^* \tau^*$, we have $\pi^* q^* \Db(\PP^1) = \chi^* \tau^* \Db(\PP^1)$ and we will
denote this category simply by $\Db(\PP^1)$.

Consider the decomposition (\ref{decomp-conic-hyperell}) and mutate $\Phi \Db(S,\kb_0)$ to the right through
$\Db(\PP^1)\otimes\ko(-h)$. The functor $\Phi' := \Phi \circ R_{\Db(\PP^1)\otimes\ko(-h)}$ is full and faithful
and we have the semiorthogonal decomposition
$$\Db(X) = \langle \Db(\PP^1)\otimes\ko(-h),\ \Phi' \Db(S,\kb_0),\ \Db(\PP^1) \rangle.$$
Perform the right mutation of $\Db(\PP^1)\otimes\ko(-h)$ through its right orthogonal. We have 
$$R_{<\Db(\PP^1)\otimes\ko(-h)>^{\perp}} (\Db(\PP^1)\otimes\ko(-h))
= \Db(\PP^1)\otimes\ko(-h) \otimes \omega_X^* \cong \Db(\PP^1)\otimes(H).$$
Indeed $\omega_X = \ko(-H - h) + \chi^* \tau^* \omega_{\PP^1}$ and the
tensorization with $\chi^* \tau^* \omega_{\PP^1}$ gives an autoequivalence of $\Db(\PP^1)$. We then have the
decomposition
$$\Db(X) = \langle \Phi' \Db(S,\kb_0),\ \Db(\PP^1),\ \Db(\PP^1)\otimes\ko(H) \rangle.$$
Comparing this last decomposition with (\ref{decomp-conic-hyperell}) we get
$$\Phi' \Db(S,\kb_0) = \langle \Psi_1 \Db(\Gamma_1), \Psi_0 \Db(\Gamma_0)\rangle.$$
\end{proof}

\begin{remark}\label{rem-the-same}
Remark that the choice of blowing up first $\Gamma_0$ and then $\Gamma_1$ has no influence (up to equivalence)
on the statement of Proposition \ref{prop:hyperell-case}.
\end{remark}

\subsection{A non-rational example}\label{non-rat-ex}
Theorem \ref{main-theorem1} states that if $\pi: X \to S$ is standard and $S$ rational, then a semiorthogonal
decomposition of $\Db(S,\kb_0)$ (and then of $\Db(X)$) via derived categories of curves and exceptional objects allows to reconstruct
the intermediate Jacobian $J(X)$ as the direct sum of the Jacobians of the curves. It is clear by the technique used, that
$S$ being rational is crucial. Using the construction by Casnati \cite{casnati}, we provide here examples of standard conic bundles
$\pi: X \to S$ over a non-rational surface such that both $\Db(X)$ and $\Db(S,\kb_0)$ admit a decomposition via derived
categories of smooth projective curves. In these cases, $X$ is clearly non-rational, and
$J(X)$ is only isogenous to $P(\tilde{C}/C) \oplus A^2(S) \oplus A^1(S)$ \cite{beltrachow}.

Let $G$ be a smooth projective curve of positive genus. Remark that $\Db(G)$ contains no exceptional object,
because of Serre duality. Consider a smooth degree four cover $\Gamma \to G$, and its embedding in a $\PP^2$-bundle
$Z \to G$. By \cite{casnati}, there is a unique
degree 3 cover $C \to G$ embedded in a ruled surface $S \to G$, and we suppose that $C$ has at most double points.
As in \ref{subsect-trigonal}, we end up with a commutative diagram:

$$\xymatrix{
& X \ar[ld]_{\pi} \ar[rd]^{\chi} & D \ar@{_{(}->}[l] \ar[rd]^{\chi}\\
S \ar[rd]_q & & Z \ar[ld]^p & \Gamma \ar@{_{(}->}[l] \\
& G,
}\label{nonratdiagram}$$

where $X$ is the blow-up of $Z$ along $\Gamma$, $D$ the exceptional divisor and $\pi: X \to S$ a standard conic bundle
degenerating along $C$, induced by the relative linear system $| \ko_{Z/G}(2) - \Gamma|$. Orlov formula for blow-ups (see Prop.
\ref{decomp-blow-up}) and Prop. \ref{prop-decom-projbund} give a semiorthogonal decomposition
$$\Db(X) = \langle \Psi \Db(\Gamma), \Db(G)\otimes(-H),\Db(G),\Db(G)\otimes(H) \rangle,$$
where we keep the notation of \ref{subsect-trigonal} and we write $\Db(G) := \chi^* p^* \Db(G) = \pi^* q^* \Db(G)$.
Then $\Db(X)$ is decomposed by derived categories of smooth projective curves.
Going through the proof of Proposition \ref{prop:trig-case}, it is clear that replacing $\PP^1$ with $G$ does not affect any calculation,
except the fact that $\Db(G)$ contains no exceptional object. Keeping the same notation, we end up with the semiorthogonal decomposition

$$\Phi'\Db(S,\kb_0) = \langle \Psi\Db(\Gamma), \Db(G) \otimes(-H) \rangle.$$


\begin{thebibliography}{9}

\bibitem{almeida-gruson-perrin}
J. d'Almeida, L. Gruson, and N. Perrin,
{\em Courbes de genre 5 munies d'une involution sans point fixe},
J. London Math. Soc. {\bf 72} (2005), no. 3, 545--570.

\bibitem{angel-mstach}
P. L. del Angel, and S. M\"uller-Stach,
{\em Motives of uniruled 3-folds}, Comp. Math. {\bf 112} (1998), 1--16.

\bibitem{artin-mumford}
M. Artin, and D. Mumford,
{\em Some elementary examples of unirational varieties which are not rational},
Proc. London math. soc. {\bf (3) 25} (1972), 75--95.

\bibitem{beauvilleprym}
A. Beauville,
{\em Prym varieties and Schottky problem},
Inventiones Math. {\bf 41} (1977), 149--196. 

\bibitem{beauvillejaco}
A. Beauville,
{\em Vari\'et\'es de Prym et jacobiennes interm\'ediaires},
Ann. scient. ENS {\bf 10} (1977), 309--391.

\bibitem{beauville-deter}
A. Beauville,
{\em Determinantal Hypersurfaces},
Michigan Math. J. {\bf 48} (2000), 39--64.

\bibitem{beilinson}
A. A. Beilinson, {\em The derived category of coherent sheaves on $\PP^n$}, Sel.
Math. Sov. {\bf 34} (1984), no. 3, 233--237

\bibitem{beltra-francia}
M. Beltrametti, and P. Francia,
{\em Conic bundles on non-rational surfaces}, in {\em Algebraic Geometry - Open problems (proceedings, Ravello
1982)}, Lect. Notes in Math. {\bf 997}, Springer Verlag, 
34--89.

\bibitem{beltrachow}
M. Beltrametti,
{\em On the {C}how group and the intermediate Jacobian of a conic bundle},
Annali di Mat. pura e applicata {\bf 41} (1985), no. 4, 331--351.

\bibitem{marcellocurves}
M. Bernardara,
{\em Fourier--{M}ukai transforms of curves and principal polarizations},
C. R. Acad. Sci. Paris, Ser. I, {\bf 345} (2007), 203--208.

\bibitem{noi-cubic}
M. Bernardara, E. Macr\'i, S. Mehrotra, and P. Stellari
{\em A categorical invariant for cubic threefolds}, Adv. Math. {\bf 229} no. 2 (2012), 770-803.

\bibitem{micheleconicbundle}
M. Bolognesi,
{\em A conic bundle degenerating on the Kummer surface},
Math. Z. {\bf 261} (2009), no. 1, 149--168.

\bibitem{bondal-repr}
A. Bondal, {\em Representations of associative algebras and coherent sheaves},
Math. USSR-Izv. {\bf 34} (1990), no. 1, 23--42.

\bibitem{bondal-kapranov}
A. Bondal, and M. Kapranov, {\em Representable functors, Serre functors, and reconstructions},
Math. USSR-Izv. {\bf 35} (1990), no. 3, 519--541.

\bibitem{bondal-orlov}
A. Bondal, and D. Orlov, {\em Semiorthogonal decomposition for algebraic varieties},
preprint math.AG/9506012.

\bibitem{casnati-ekedhal}
G. Casnati, and T. Ekedahl,
{\em Covers of algebraic varieties. I. A general structure theorem, covers of degree $3,4$ and Enriques surfaces}, 
J. Algebraic Geom. {\bf 5} (1996), no. 3, 439--460.

\bibitem{casnati}
G. Casnati,
{\em Covers of algebraic varieties III: The discriminant of a cover of degree 4 and the trigonal construction},
Trans. Am. Math. Soc. {\bf 350} (1998), no. 4, 1359--1378.

\bibitem{clem-griff}
C.H. Clemens, and P.A. Griffiths,
{\em The intermedite Jacobian of the cubic threefold}, Ann. Math. (2) {\bf 95}, 281--356 (1972).

\bibitem{dolgaclassic}
I.Dolgachev,
{\em Topics in classical algebraic geometry},
Lecture Notes available at http://www.math.lsa.umich.edu/~idolga/topics.pdf.

\bibitem{donagi-prym}
R. Donagi,
{\em The fibers of the Prym map}, in {\em Curves, Jacobians, and Abelian varieties, Proc. AMS-IMS-SIAM Jt.
 Summer Res. Conf. Schottky Probl.}, Contemp. Math. {\bf 136} (1992), 55--125.

\bibitem{bertbrau}
B. van Geemen
{\em Some remarks on Brauer groups of $K3$ surfaces.},
Adv. in Math. {\bf 197} (2005), 222--247.

\bibitem{harris-first-course}
J. Harris,
{\em Algebraic geometry. A first course},
Graduate Texts in Math. {\bf 133}, Spinger-Verlag (1992).

\bibitem{huybrechts}
D. Huybrechts,
{\em Fourier-{M}ukai transforms in {A}lgebraic {G}eometry},
Oxford Math. Monongraphs (2006).

\bibitem{iliev-marku}
A. Iliev, and D. Markushevich,
{\em The Abel-Jacobi map for cubic threefold and periods of Fano threefolds of degree 14},
Doc. Math., J. DMV 5, 23--47 (2000).

\bibitem{iskocongru}
V. A. Iskovskikh, 
{\em Congruences of conics in $P^{3}$} (Russian. English summary), 
Vestnik Moskov. Univ. Ser. I Mat. Mekh. 1982, no. 6, 57--62, 121;
English translation: Moscow Univ. Math. Bull. {\bf 37} (1982), no. 6, 67--73.

\bibitem{iskovconicduke}
V.A. Iskovskikh,
{\em On the rationality problem for conic bundles},
Duke Math. J. {\bf 54} (1987), 271--294

\bibitem{kapranovquadric}
M.M. Kapranov,
{\em On the derived categories of coherent sheaves on some homogeneous spaces},
Invent. Math. {\bf 92} (1988), 479--508.

\bibitem{katzarkov1}
L. Katzarkov, {\em Generalized homological mirror symmetry, superschemes and nonrationality},
in {\em Special metrics and supersymmetry. Lectures given in the workshop on geometry and physics:
special metrics and supersymmetry, Bilbao, Spain, 29--31 May 2008}, AIP Conference Proceedings {\bf 1093}, 92--131 (2009).

\bibitem{katzarkov2}
L. Katzarkov,
{\em Generalized homological mirror symmetry and rationality questions}, in
{\em Cohomological and geometric approaches to rationality problems},
Progr. Math. {\bf 282}, Birkh\"auser Boston, 163--208 (2010).

\bibitem{kuznet-v14}
A. Kuznetsov,
{\em Derived categories of cubic and $V_{14}$ threefolds}
Proc. Steklov Inst. Math. {\bf 246}, 171--194 (2004); translation from Tr. Mat. Inst. Steklova {\bf 246}, 183--207 (2004). 

\bibitem{kuznetconicbundles}
A. Kuznetsov,
{\em Derived categories of quadric fibrations and intersections of quadrics},
Adv. Math. {\bf 218} (2008), no. 5, 1340--1369.

\bibitem{kuznet-cubic4folds}
A. Kuznetsov,
{\em Derived categories of cubic fourfolds},
in {\em Cohomological and geometric approaches to rationality problems},
Progr. Math. {\bf 282}, Birkh\"auser Boston, 163--208 (2010).

\bibitem{kuznet-fano3folds}
A. Kuznetsov,
{\em Derived categories of Fano threefolds}, Proc. Steklov Inst. Math. {\bf 264} (2009), no.1 110--122.

\bibitem{merkurev}
A. S. Merkur'ev, {\it On the norm residue symbol of degree 2}, Dokl. Akad. Nauk SSSR {\bf 261}, No. 3 (1981), 542--547,
english translation in Soviet Math. Doklady {\bf 24} (1981), 546--551.

\bibitem{mumfordprym}
D. Mumford,
{\em Prym Varieties I},
in {\em Contribution to Analysis},
Academic Press, 1974.

\bibitem{murre}
J.P. Murre, {\em On the motive of an algebraic surface}, J. reine angew. Math. {\bf 409} (1990), 190--204.

\bibitem{nagel-saito}
J. Nagel, and M. Saito,
{\em Relative {C}how-{K}\"unneth decomposition for conic bundles and {P}rym varieties},
Int. Math. Res. Not. {\bf 2009}, no.16, 2978--3001.

\bibitem{orlovprojbund}
D.O. Orlov,
{\em Projective bundles, monoidal transformations and derived categories of coherent sheaves},
Russian Math. Izv. {\bf 41} (1993), 133--141.

\bibitem{orlov-represent}
D.O. Orlov,
{\em Derived categories of coherent sheaves and equivalences between them},
Russian Math. Surveys {\bf 58} (2003), 511--591.

\bibitem{orlov-motiv}
D.O. Orlov, {\em Derived categories of coherent sheaves and motives},
Russian Math. Surveys {\bf 60} (2005), 1242--1244.

\bibitem{panin}
I.A. Panin,
{\em Rationality of bundles of conics with degenerate curve of degree five and even theta-characteristic},
J. Sov. Math. {\bf 24} (1984), 449-452;
Russian original in Zap. Nauchn. Semin. Leningr. Otd. Mat. Inst. Steklova {\bf 103} (1980), 100-106.

\bibitem{recillas}
S. Recillas,
{\em Jacobians of curves with $g^1_4$'s are the Pryms of trigonal curves},
Bol. Soc. Mat. Mex., II. Ser. {\bf 19} (1974), 9--13.

\bibitem{sarkisov-bira}
V.G. Sarkisov,
{\em Birational automorphisms of conic bundles},
Math. USSR, Izv. {\bf 17} (1981), no. 1, 177--202.

\bibitem{sarkisov-struct}
V.G. Sarkisov,
{\em On conic bundle structures}, Math. USSR Izv. {\bf 20}, No. 2 (1982), 355--390.

\bibitem{scholl}
A.J. Scholl,
{\em Classical motives}, in {\em Motives. Proceedings of the summer research conference on motives,
held at the University of Washington, Seattle, WA, USA, July 20-August 2, 1991}, AMS Proc. Symp. Pure Math.
{\bf 55}, (1994) Pt. 1, 163-187 (1994).

\bibitem{shokuprym}
V.V. Shokurov,
{\em Prym varieties: theory and applications},
Math. USSR-Izv. {\bf 23} (1984), 83--147.

\bibitem{verraprym}
A. Verra,
{\em The fibre of the Prym map in genus three},
Math. Ann. {\bf 276} (1987), 433--448.


\end{thebibliography}
\end{document}